\documentclass[a4paper,11pt]{article}
\expandafter\let\csname equation*\endcsname\relax
\expandafter\let\csname endequation*\endcsname\relax

\usepackage{amsmath}
\usepackage{amsthm}
\usepackage{amssymb}
\usepackage{graphicx}
\usepackage{epstopdf}
\usepackage{algorithm}
\usepackage{algorithmic,color}
\usepackage{multirow}
\usepackage{appendix}
\usepackage{color}
\usepackage{comment}
\usepackage{tikz}
\usepackage{rotating}
\newcommand*\circled[1]{\tikz[baseline=(char.base)]{
            \node[shape=circle,draw,inner sep=2pt] (char) {#1};}}

\newcommand{\R}{{\mathbb R}}

 \def\prox{\mathrm{prox}}

 \def\d{\,\mathrm{d}}
\def\argmin{\mathop{\mathrm{arg min}}}

\usepackage{mathrsfs}
\usepackage{xcolor}
\usepackage{enumerate}

\def\Tau{\mathscr{T}}
\def\Wt{{W_{\Tau}}}
\def\Vt{{V_{\Tau}}}
\def\pt{{p_{\Tau}}}
\def\Gammat{{\Gamma_{\Tau}}}
\def\nt{{n}}
\def\nM{m}
\def\mp{{m_p}}
\def\sigmaL{{\sigma_\lambda^\delta}}
\def\sigman{{\sigma^{(n)}}}
\def\sigmanp{{\sigma^{(n+1)}}}
\def\Ld{{\Lambda^\delta}}
\newcommand{\upchi}{{\text{\raisebox{1.5pt}{$\chi$}}}}

\providecommand\argmin{\mathop{\mathrm{argmin}}}

\newtheorem{theorem}{Theorem}

\newtheorem{prop}{Proposition}
\newtheorem{cor}{Corollary}
\newtheorem{rem}{Remark}

\usepackage[a4paper,left=2.cm,right=2.cm,top=3.cm,bottom=3.cm]{geometry}

\usepackage[a4paper,left=2.cm,right=2.cm,top=3.cm,bottom=3.cm]{geometry}

\begin{document}

\title{Oracle-Net for nonlinear compressed sensing in Electrical Impedance Tomography reconstruction problems}

\date{ }
\author{Damiana Lazzaro\thanks{ ORCID ID 0000-0002-2029-9842. Department of Mathematics, University of Bologna, Bologna, Italy. Email: damiana.lazzaro@unibo.it. }
      \and
    Serena Morigi\thanks{ORCID ID 0000-0001-8334-8798. Department of Mathematics, University of Bologna, Bologna, Italy.
   Email: serena.morigi@unibo.it.}
  \and
    Luca Ratti\thanks{ORCID ID 0000-0001-7948-0577. Department of Mathematics, University of Bologna, Bologna, Italy. 
    Email: luca.ratti5@unibo.it. Corresponding author.}
}

\maketitle

\begin{abstract}
Sparse recovery principles play an important role in solving many nonlinear ill-posed inverse problems. 
We investigate a variational framework with support Oracle for compressed sensing sparse reconstructions, where the available measurements are nonlinear and possibly corrupted by noise.
A graph neural network, named Oracle-Net, is proposed to predict the support from the nonlinear measurements and is integrated into a regularized recovery model to enforce sparsity. The derived nonsmooth optimization problem is then efficiently solved through a constrained proximal gradient method.
Error bounds on the approximate solution of the proposed Oracle-based optimization are provided in the context of the ill-posed Electrical Impedance Tomography problem.   Numerical solutions of the EIT nonlinear inverse reconstruction problem confirm the potential of the proposed method which improves the reconstruction quality from undersampled measurements, under sparsity assumptions. 
\end{abstract}
\bigskip

\section{Introduction}

The theory of compressed sensing (CS) is a successful  mathematical technology in sparse signal recovery, established several years ago by Donoho \cite{Donoho2006}, successfully carried out by \cite{CW2008,CRT2006}, 
and widespreadly developed in the last two decades.
The main idea behind CS is to acquire a small number of linear measurements of a signal that exhibits sparsity or compressibility in a known representation system, and accurately reconstruct the original sparse signal using an appropriate reconstruction algorithm, under the assumption that the sensing matrix used in acquisition
satisfies the Restricted Isometry Property (RIP).
The standard recovery algorithms for finding an approximated CS solution are based either on a greedy approach or on variational models, such as $\ell_1$-norm minimization, leading to suitable iterative thresholded gradient descent methods.

A significant portion of research focuses on {\em linear} CS problems and holds for {\em well-posed} compressive sensing models. However many real-world applications in physics and biomedical sciences involve inherent nonlinearities. In these cases, the linear model becomes inadequate.

We are interested in a more general setting that extends the concepts of compressive sensing and sparse recovery to inverse and {\em ill-posed} {\em nonlinear} problems.

In this more general setting, we consider the problem of recovering an unknown vector $\sigma^{\dag} \in \R^n$  from incomplete and contaminated nonlinear measurements $\Lambda^\delta \in \R^m$, according to the following degradation model
\begin{equation}
\label{eq:cs}
\Lambda^\delta = \Phi(\sigma^\dag)+\eta,
\end{equation}
where $\eta \in \R^m$ is a vector of unknown perturbations, bounded by a known amount $\|\eta\|_2^2\le \delta^2$, and $\Phi: \R^n \rightarrow \R^m$ represents a nonlinear ill-posed sensing model characterized by an undersampled acquisition in which the number $m$ of available measurements is much smaller ($m \ll n$) than the dimension $n$ of the vector $\sigma$.

We moreover assume that, although $\sigma$ belongs to the high-dimensional space $\R^\nt$, it can be represented by a few degrees of freedom. In doing so, we take a step forward from the classical sparsity assumption and search for conductivities $\sigma$ that are $\sigma_0$-sparse, being $\sigma_0$ a reference constant vector. This accounts to say that $\sigma$ is equal to $\sigma_0$ apart from a few coefficients, thus the set of indices of $I_{\mathcal{O}}=\{ i : (\sigma_i-\sigma_{0,i})  \neq 0 \}$ is assumed to have small cardinality $s$ (roughly $O(m)$). 
In the context of CS, several algorithms have been proposed to enforce (or promote) this property, based on projection and variational techniques. By analogy with the CS linear case, and motivated furthermore by the ill-posed nature of $\Phi$ (which typically prevents problem \eqref{eq:cs} from having a unique and stable solution) we propose to recover the vector $\sigma^\dag$ by solving the following regularized minimization problem:
\begin{equation}
\min _{\sigma \in \R^{n}} \mathcal{R}(\sigma) \, \quad \mbox{subject to} \quad \, \|\Phi(\sigma) - \Lambda^\delta\|_2^2 \le \delta^2,
\label{eq:prob}
\end{equation}
where $\mathcal{R} \in \Gamma_0(\R^n)$ is a sparsity-promoting penalty and we denote by $\Gamma_0(\R^n)$ the space of proper, lower semi-continuous, and convex funcionals on $\R^n$. 
In particular, we will consider functionals $\mathcal{R}$ that pursue sparsity promotion in two different, complementary, ways: by promoting low values of a penalization functional and by projecting on a predetermined set $K$. 
The most standard instance in the class of nonsmooth, convex penalties is the $\ell_1$-norm functional, which is known to promote sparsity in the solution domain; nevertheless, we may also opt for other choices, such as the Total Variation penalty \cite{RUDIN1992259}, to enforce other significant prior knowledge on the solution, such as sparsity in the gradient domain.
On the other hand, the projection onto the set $K$ can be employed to impose state constraints (such as non-negativity, or box constraint) or to enforce sparsity likewise. 
The approach we propose in this paper draws connections with the problem known as \textit{support estimation}, namely, with the task of determining the support of the exact solution from the given measurements. We thus introduce the concept of \textit{Oracle}: an operator receiving as an input the 
measurements $\Lambda^\delta$ and returning a (perfect or approximate) description of its support $K$, for example through the set of active indices $I_{\mathcal{O}}$. 
Such an Oracle operator must be determined before the resolution of problem \eqref{eq:prob}: in this paper, we pursue this task in a supervised statistical learning fashion, by training a Neural Network via a dataset of ground truth solutions (hence, their supports) and associated measurements.
We consider, in particular, Graph Neural Networks (GNNs)
based on a graph U-Net architecture, which naturally fits the mesh data structure used to represent the spatial domain. GNNs are a class of neural network architecture designed to perform inference on data described by graphs \cite{ZHOU202057}. The benefit of using GNN is related to the specific nonlinear problem considered as a case test of the proposed CS framework, which requires a FEM solution on a mesh domain.
Once the Oracle has been learned, we can employ the set $K$ within the definition of $\mathcal{R}$ and tackle the solution of problem \eqref{eq:prob}.

Independently of the choice of $\mathcal{R}$, our strategy then envisages the solution of \eqref{eq:prob}
using a first-order optimization scheme, namely, the Proximal Gradient Method (PGM).
Despite the regularizing effect of $\mathcal{R}$, because of the nonlinear nature of $\Phi$, the convergence properties of such a scheme are only guaranteed under additional assumptions.

The main theoretical results of this paper are reported in Proposition \ref{prop:conv_rate} and in Theorem \ref{thm:blumensath}. The first one shows that minimizer of the proposed regularization functional converges to a solution of the inverse problem \eqref{eq:cs} as the noise level $\delta$ goes to zero, also providing some convergence rates. 
The second result concerns error bounds on the approximate solution obtained via the proposed Oracle-based PGM. Drawing from a result in \cite{blumensath2013compressed}, we show the convergence of such a scheme to a cluster point whose distance from an exact solution of \eqref{eq:cs} can be bounded in terms of $\delta$ (and, possibly, of the choice of the support $K$).
In order to derive similar error bounds in the context of linear CS, it is usually necessary to introduce theoretical assumptions on the forward operator, such as the RIP, which do not carry over directly to the nonlinear case. 
In this work, we show that, under the requirements that the Jacobian of the measurement system $\Phi$ satisfies a similar RIP condition and that $\Phi$ is mildly nonlinear, the accurate recovery of $\sigma^\dag$ is possible using the proposed Oracle-based PGM algorithm.

Due to the many difficulties of this challenging CS recovery problem, we will focus on a specific nonlinear measurement process that is involved in the Electrical Impedance Tomography (EIT), a promising non-invasive imaging technique mathematically formulated as a highly nonlinear ill-posed inverse problem \cite{Somersalo}. 
Nevertheless, our Oracle-based framework could be adapted to other nonlinear ill-posed problems.
EIT aims to reconstruct the inner conductivity distribution of a medium starting from a set of measured voltages registered by a series of electrodes that are
positioned on the surface of the medium. EIT is therefore a nondestructive testing technique, meaning that it allows analyzing the inner properties of a material or structure without causing damage.

Several CS strategies have been explored in the context of linearized EIT reconstruction problem, see \cite{Teh, GEHRE}, as a response to a growing application need.
By leveraging CS, in \cite{Shiraz} the authors explore the possibility of achieving accurate breathing monitoring with EIT while reducing the number of measurements needed.
 The linearized EIT problem is considered from a theoretical perspective in \cite{Alberti_2021} where the authors show that the electrical conductivity may be stably recovered from a number of linearized EIT measurements proportional to the sparsity of the signal with respect to a wavelet basis, up to a log factor.
 
Real-world EIT data exhibits nonlinearities. To address this challenge, Zong et al. in \cite{Zong} recently proposed a novel scheme that introduces the concept of compressive learning. 
Recently, a learned residual approach to EIT has been proposed in \cite{SSVM2023} for nonlinear undersampled measurements.
In \cite{Chen2022} a mask-guided
spatial–temporal graph neural network is proposed to reconstruct multifrequency EIT images in cell culture imaging. The binary masks in \cite{Chen2022} are obtained by CT scans and microscopic images using a multimodal imaging setup. 

In this work, the concept of compressed learning is realized through an Oracle-Net that predicts the optimal prior in the variational model which improves the efficacy of the sparse optimization algorithm.
The binary masks are thus automatically determined by the Oracle-Net starting from given measurements.
We finally would shed light on the problem of determining how few measurements suffice for
an accurate EIT sparsity-regularized reconstruction.

\vspace{0.3cm}

The remainder of this paper is structured as follows. In Section \ref{sec:sec2} related works on nonlinear compressed sensing are presented.
Section \ref{sec:sec3} presents the EIT forward and inverse models in their continuous settings. 
In Section \ref{sec:sec4} we provide general theoretical results on the well-posedness of the regularized problem, we introduce the PGM-based numerical method for EIT reconstruction, and we prove error bounds on the sparsity-aware approximate solution.
Section \ref{sec:sec5} examines some specific regularization models employed and discusses their algorithmic optimization. 
Section \ref{sec:sec6} explores the architecture of the proposed Oracle-Net for support estimation. In Section \ref{sec:sec7} we present the reconstruction results obtained by the Oracle-based PGM algorithm. Finally,  Section \ref{sec:sec8} draws conclusions and discusses future work.

\section{Related works}
\label{sec:sec2} 
 The treatment of linear compressed measurement models for ill-posed problems and recovery accuracy estimates has been first addressed in \cite{Herrholz_2010}.

A preliminary exploration of nonlinear CS can be found in the work on CS phase retrieval  \cite{Baraniuk} and on 1-bit CS \cite{Bar_1B}. 
A first attempt to extend greedy gradient-based strategies to the nonlinear case is in \cite{B2008}. Subsequently, in \cite{blumensath2013compressed}
 Blumensath pioneered the theoretical framework for nonlinear compressed sensing. It has been demonstrated that, under conditions similar to the RIP, the iterative hard thresholding (IHT) algorithm can effectively recover sparse or structured signals from a small number of nonlinear measurements.
Along this direction, further researches introduced various approaches for tackling nonlinear CS problems. In  \cite{Ohlsson2014,O2013} Ohlsson et al. proposed algorithms like quadratic basis pursuit and nonlinear basis pursuit, both of which utilize convex relaxations to solve these problems. In addition, research on sparsity-constrained nonlinear optimization, as explored in \cite{BeckEldar}, further enriched the field of nonlinear CS.
A generalization of the RIP condition for certain randomized quasi-linear measurements is proposed in \cite{Fornasier}.

The power of nonlinear CS extends beyond theoretical developments. Practical applications validate its effectiveness. For example, works like \cite{Klodt} demonstrate its application in single-snapshot compressive tomosynthesis, while the proposal in \cite{OPC}  leverages the inherent sparsity of mask patterns to formulate the Optical Proximity Correction problems as inverse nonlinear CS problems, enabling significant efficiency gains.

The aforementioned works seem to provide a CS solution to nonlinear inverse problems neglecting the effect of the ill-posedness of several nonlinear problems. In this work, the combined effect of a sparsifying Oracle-Net and an effective regularization guarantees accurate reconstructions of the original sparse signal using a limited number of nonlinear measurements.

\section{EIT Forward and Inverse models}
\label{sec:sec3}

In this section, we describe the EIT problem in its original continuous setup, focusing on the so-called \textit{complete electrode model} (CEM), both from a forward and an inverse perspective. Our goal is to clearly state the problem of interest in a functional space framework, as well as to motivate the discretization that will be later adopted. Moreover, we collect in Theorem \ref{thm:Lechleiter} the most relevant theoretical results regarding the nonlinear operator describing the problem, mainly relying on the approach and assumptions adopted in \cite{lechleiter2008newton}.
\par
Consider a conductive body in a 2-dimensional space, described as a bounded, simply connected, Lipschitz domain $\Omega \subset \mathbb{R}^2$ with a piece-wise $C^\infty$ boundary $\partial \Omega$. Its electrical conductivity is modeled as a function $\sigma \colon \Omega \rightarrow \R$, which we assume to be bounded, measurable, and larger than a strictly positive constant $c_0$:
\begin{equation}
\sigma \in L^{\infty}(\Omega), \qquad \sigma(x) \geq c_0 > 0 \quad \text{ a.e. in $\Omega$}.
    \label{eq:sigma}
\end{equation}
The classical formulation of the EIT problem, also known as \textit{continuum model}, seeks to reconstruct the function $\sigma$ by applying electrical stimuli on the boundary $\partial \Omega$ (i.e., injected currents), and recording the resulting electrical potential, again on the whole boundary $\partial \Omega$. A more realistic scenario is described via CEM, in which one assumes that the process of probing the electrical properties of $\Omega$ is performed by means of some electrodes located on the boundary of the domain.
Consider a collection of $p$ electrodes $\{E_j \}_{j=1}^p \subset \partial \Omega$: from a modeling perspective, they are resistive regions (with electrical permittivity $z$), on which it is possible to apply an external current $I_j$. We assume that each $E_j$ is an open, non-empty subset of $\partial \Omega$, and that $E_i \cap E_j = \emptyset \text{ for $i \neq j$}$. The electrical measurements, both of currents and voltages, associated with the CEM are assumed to be piece-wise constant functions on $\partial \Omega$: we define the electrode space as 
\begin{equation}
    \mathcal{E}_p = \left\{F \in L^2(\partial \Omega): \quad F(x) = \sum_{j=1}^p F_j \mathbf{1}_{E_j}(x), \quad \int_{\partial \Omega} F(y)dy = \sum_{j=1}^p F_j |E_j| = 0 \right\},
    \label{eq:electrode_space}
\end{equation}
where $\mathbf{1}_{E_j}$ denotes the indicator function of $E_j \subset \partial \Omega$. Notice that $F \in \mathcal{E}_p$ can be uniquely determined by means of $p-1$ real parameters, thus $\mathcal{E}_p$ can be equivalently represented as $\R^{p-1}$. We denote by $E = \bigcup_{j=1}^p E_j$ and by $\Gamma = \partial \Omega \setminus E$.
\par
The forward problem of EIT, in the CEM paradigm, consists in determining the couple $(u,U) \in Y = H^1(\Omega) \times \mathcal{E}_p$, representing the inner electrical potential and the boundary voltages at the electrodes, associated with the input boundary currents $I \in \mathcal{E}_p$ through the following differential problem:
\begin{equation}
\label{CEM}
\left\{
\begin{aligned}
- \operatorname{div} (\sigma \nabla u)&= 0 \quad  & \mbox{in} \  \Omega, \\
u + z \sigma \frac{\partial u}{\partial n}&= U_l & \mbox{on} \  \{E_j\}_{j=1}^p, \\
\int_{E_l}{\sigma \frac{\partial u }{\partial n} \d s} &= I_l & \mbox{$j=1,\ldots,p$}, 
\\
\sigma \frac{\partial u}{\partial n} &=0 &  \mbox{on}  \  \Gamma.
\end{aligned}
\right.
\end{equation}
Here $z$ represents the known contact impedance of the electrodes.
The existence and uniqueness of a weak solution of the boundary problem \eqref{CEM} is proved in \cite{Somersalo}, and its numerical approximation is effectively treated, e.g., via finite element method (FEM). 
\par
Denote by $\mathcal{L}(\mathcal{E}_p)$ the space of bounded linear operators from $\mathcal{E}_p$ to $\mathcal{E}_p$ (isomorphic to the matrix space $\R^{(p-1)\times(p-1)}$).
For a fixed conductivity $\sigma$ satisfying \eqref{eq:sigma}, the well-posedness result implies that, for every choice of input boundary current $I \in \mathcal{E}_p$, the couple $(u,U) \in Y$ is uniquely determined, and we can denote by $\Lambda^\delta_\sigma$ the operator that associates a current $I \in \mathcal{E}_p$ with the corresponding boundary voltage $U \in \mathcal{E}_p$. We can thus define the \textit{forward map} of EIT, namely, the operator $F_p \colon L^\infty(\Omega) \rightarrow \mathcal{L}(\mathcal{E}_p)$, which returns, for each conductivity $\sigma$, the currents-to-voltages operator $\Lambda^\delta_\sigma$:
\begin{equation}
F_p \colon \sigma \in L^{\infty}(\Omega) \ \mapsto\ \Lambda^\delta_\sigma = F_p(\sigma) \in \mathcal{L}(\mathcal{E}_p): \quad \Lambda^\delta_\sigma I = U \ \text{ s.t. $(u,U)\ $ solves \eqref{CEM}}.
    \label{eq:F_p}
\end{equation}
Notice that for each $\sigma$, the operator $\Lambda^\delta_\sigma$ can be identified by a square matrix in $\R^{(p-1)\times(p-1)}$. Moreover, since the problem \eqref{CEM} is symmetric in the variables $U$ and $I$, we can efficiently represent $\Lambda_\sigma^\delta$ as a symmetric matrix, or as a vector in $\R^{\mp}$, being $\mp = \frac{p(p-1)}{2}$ the number of independent measurements associated with $p$ electrodes. 
\par
The \textit{inverse problem} of EIT is to reconstruct $\sigma$ from the knowledge of $\Lambda^\delta_\sigma$. Observe that, since the available measurements are only finite-dimensional, we can only hope to recover conductivities depending on a finite number of degrees of freedom. 
In particular, we introduce a conformal triangular partition $\Tau$ over $\Omega$ and define the Finite Element space of piecewise affine, continuous, and uniformly bounded functions
\begin{equation}
\Wt = \bigg\{ v\in C(\overline{\Omega})  \quad \text{s.t.} \quad v|_{T_k} \in \mathbb{P}_1(T)\ \ \forall T \in \Tau \bigg\},
    \label{eq:W_tau}
\end{equation}
where $\mathbb{P}_1(T)$ denotes the space of polynomials of degree $1$ on the triangle $T$. This choice fits the assumptions in \cite{lechleiter2008newton}, where the partition can also be non-triangular and non-conformal and the functions in $\Wt$ are bounded piece-wise polynomials of arbitrary degree.

A $\nt$-dimensional basis can be defined on $\Wt$, where $\nt$ is the number of vertices in the partition $\Tau$, by means of the piece-wise affine functions $\{\phi_i\}_{i=1}^\nt$, each of which attains the values $1$ at the $i$-th vertex and $0$ at all the other ones. A function in $\Wt$ is thus determined by its values at the vertices, $\{v_i\}_{i=1}^n$, which implies that $\Wt$ can be identified with $\R^{\nt}$, and the forward map $F_p$ restricted on $\Vt$ can be interpreted as a nonlinear function from $\R^{\nt}$ to $\R^{\mp}$. Its Fréchet derivative with respect to the $L^\infty$ norm, denoted by $F'_p$ can therefore be represented as a $\mp \times \nt$ matrix, the Jacobian matrix of the described vector field. 
\par
In \cite{lechleiter2008newton}, several results regarding the map $F_p$ (and its derivative $F'_p$) can be found. We are nevertheless interested in the properties of a strongly related map, $F_{p,h}$, which arises by a further discretization of the problem, induced by the numerical approximation of the differential problem \eqref{CEM}. In particular, as in \cite{lechleiter2008newton}, we introduce a computational mesh $\Tau_h$ on $\Omega$, in general unrelated to the previously introduced $\Tau$, and consisting of $N_h$ triangles. We approximate the differential problem \eqref{CEM} through a Finite Element scheme, 
and introduce the operator $\Lambda^\delta_{\sigma,h} \in \mathcal{L}(\mathcal{E}_p)$, which maps every $I_h \in \mathcal{E}_p$ into the potential $\Lambda^\delta_{\sigma,h}I_h = U_h$ such that $(u_h,U_h)$ is the output of the Finite-Element solver of \eqref{CEM}. Then, we can define $F_{p,h}$ as the map from $\Wt$ to $\mathcal{L}(\mathcal{E}_p)$ that associates a conductivity $\sigma$ to $\Lambda^\delta_{\sigma,h}$.
We now recap the main theoretical properties satisfied by $F_{p,h}$, which are analyzed in \cite{lechleiter2008newton}. In particular, we employ further knowledge of the desired solution to limit our search to the space
\begin{equation}
\Vt = \bigg\{ v\in \Wt  \quad \text{s.t.} \quad  c_0 \leq v(x) \leq c_1 \ \ \forall x \in \Omega \bigg\},
    \label{eq:V_tau}
\end{equation}
being $0<c_0<c_1$. These uniform bounds of the desired conductivities allow for a simpler expression for the estimates contained in \cite{lechleiter2008newton}.

\begin{theorem} (Properties of $F_{p,h}$ and $F'_{p,h}$, from \cite{lechleiter2008newton})
\label{thm:Lechleiter}
\begin{enumerate}
    \item The operator $F_{p,h}$ is Fréchet differentiable and its derivative $F'_{p,h}(\sigma)$ is Lipschitz continuous for all $\sigma \in \Vt$
    \item Local injectivity: there exist and integer $\pt$ and two positive constants $h_\Tau, \Gammat$ such that, for $p > \pt$ and $h \leq h_\Tau$, 
    \[
    \| F'_{p,h}(\sigma)[\theta]\|_{\mathcal{L}(\mathcal{E}_p)} \geq \Gammat \| \theta \|_{L^{\infty}} \quad \forall \sigma \in \Vt, \ \forall \theta \in \Wt.
    \]
    \item Tangential cone condition: for each $\sigma \in \operatorname{int}(\Vt)$, there exist a radius $r_\Tau$ and a constant $C_\Tau$ such that, for $p > \pt$ and $h \leq h_\Tau$, if $\|\tau - \sigma\| \leq r_\Tau$,
    \[
    \| F_{p,h}(\tau) - F_{p,h}(\sigma) - F'_{p,h}(\sigma)[\tau - \sigma]\|_{\mathcal{L}(\mathcal{E}_p)} \leq C_\Tau \| F_{p,h}(\tau) - F_{p,h}(\sigma)\|_{\mathcal{L}(\mathcal{E}_p)}
    \]
\end{enumerate}
\end{theorem} Note that the result on the tangential cone condition controls the linearization error by the nonlinear Taylor remainder.

The first statement follows by a combination of Lemma 2.3, Lemma 4.1, and Lemma 4.6 in \cite{lechleiter2008newton} (see also \cite{lechleiter},\cite{WANG2021},\cite{Kindermann2021OnTT}), whereas the second and the third statements, with minimal modifications, are the objects of Theorem 4.7 and Theorem 4.9 in \cite{lechleiter2008newton}, respectively. In contrast with the formulation of these results presented therein, the values of $p_\Tau, h_\Tau, \Gammat$ depend on the bounds $c_0, c_1$ on the conductivities, which is considered as a parameter in the definition of $\Vt$. Explicit expressions or bounds for $\pt,h_\Tau$ and $\Gammat$ are not available, and their definitions are not constructive, as they also involve the Fréchet derivative on the forward map of $F_p$ and of the analogous operator of the continuum model of EIT.

\section{A regularized constrained model for the EIT inverse problem}\label{sec:sec4}

In this section, we introduce a reconstruction model for the inverse problem of EIT in the CEM formulation. Thanks to the discrete nature both of $\sigma$ and of the measurements $\Lambda^\delta_{\sigma,h}$, this problem can be formulated as the resolution of a nonlinear (ill-posed) system of algebraic equations. We introduce a variational regularization strategy, which involves the minimization of a (non-convex, non-smooth) functional, for which we will describe in Section \ref{sec:PGM} an iterative scheme based on the Proximal-Gradient method \cite{Beck_Teboulle_2009}. 
\par
As already discussed in Section \ref{sec:sec3}, we look for solutions of the inverse problem of EIT in a finite-dimensional space of piece-wise constant conductivities, $\Wt$, which can be naturally identified with $\R^{\nt}$, and more precisely within $\Wt$, identified with the hyper-cube $K_{0,1}=[c_0,c_1]^{\nt}$. The datum of the inverse problem, i.e., the currents-to-voltage operator $\Lambda^\delta_{\sigma,h}$, also belongs to a finite-dimensional space, $\mathcal{L}(\mathcal{E}_{p,h})$, which can be identified with $\R^{\mp}$.
Nevertheless, in the context of applications, it is most common to provide alternative, often redundant, representations of the operator $\Lambda^\delta_{\sigma,h}$ by considering $n_c$ different current patterns (often associated with the activation of a few, adjacent or opposite, electrodes) and recording the voltage in (other) $n_v$ electrodes. All these modeling choices define the so-called measurement protocol, as well as the total number of measurements $\nM = n_c n_v$. We assume that $\nM \geq \mp$, and identify the measurement space with $\R^\nM$.  
\par
For fixed discrete mesh $\Tau$ and measurement protocol,  we denote by $\Phi\colon \R^{\nt} \rightarrow \R^{\nM}$ the nonlinear operator representing $F_{p,h}$ from $\Wt$ to the (redundant, $m$-dimensional) representation of the measurement space $\mathcal{L}(\mathcal{E}_{p,h})$. The inverse problem of EIT is thus equivalent to recovering $\sigma^\dag \in K_{0,1}=[c_0,c_1]^{\nt}$ from the noisy measurements $\Lambda^\delta \in \R^{\nM}$ under the degradation model \eqref{eq:cs}.

In this section, we analyze the variational regularization strategy associated with the following constrained optimization problem:
\begin{equation}
    \sigmaL \in \argmin_{\sigma \in K} \left\{ \mathcal{J}^\delta_\lambda(\sigma):= \frac{1}{2}\| \Phi(\sigma) - \Lambda^\delta \|^2 + \frac{\lambda \rho}{2} \| \sigma \|^2 + \lambda R(\sigma) \right\},
    \label{eq:regularization_constr}
\end{equation}
where $\rho, \lambda >0$, $R \colon \R^\nt \rightarrow \R \cup \{\infty\}$ is a non-negative, continuous, coercive, and convex functional, and $K$ is a compact, convex subset of $\R^\nt$ satisfying $K \subset K_{0,1}$. 

By introducing  the characteristic function $\upchi_K$ of the set $K \subset \R^\nt$, namely
\[
\upchi_K(x) = \left\{ \begin{aligned} 0 \quad \text{if } x \in K \\ \infty \quad \text{if } x \notin K \end{aligned} \right.\ ,
\] 
allows to incorporate the constraint $\sigma \in K$ into the minimization problem \eqref{eq:regularization}, thus obtaining the following equivalent unconstrained optimization problem 
\begin{equation}
    \sigmaL \in \argmin_{\sigma \in \R^\nt} \left\{ \mathcal{J}^\delta_\lambda(\sigma):= \frac{1}{2}\| \Phi(\sigma) - \Lambda^\delta \|^2 + \frac{\lambda \rho}{2} \| \sigma \|^2 + \lambda R(\sigma) + \upchi_K(\sigma) \right\}.
    \label{eq:regularization}
\end{equation}

The functional $\mathcal{J}_\lambda^\delta$ presents two regularization terms (namely, the square norm and the convex functional $R$). It would be possible to consider the two regularization parameters as independent: nevertheless, as this would not provide any significant difference in the theoretical analysis, similarly to the Elastic-Net paradigm \cite{zou2005regularization,de2009elastic}, we consider their ratio $\rho$ as fixed and interpret the functional as depending on a single parameter, $\lambda$.

In the following we discuss some theoretical properties of problem \eqref{eq:regularization} and of its solutions $\sigmaL$. To ease the notation, we denote the overall regularization functional employed in \eqref{eq:regularization} by
\begin{equation}
    \mathcal{R}(\sigma) = R(\sigma) + \frac{\rho}{2} \| \sigma \|^2 + \upchi_K(\sigma),
    \label{eq:glob_reg}
\end{equation}
so that the functional to be minimized can be simply written as
\[
\mathcal{J}^\delta_\lambda(\sigma) = \frac{1}{2}\|\Phi(\sigma)-\Lambda^\delta\|^2 + \lambda \mathcal{R}(\sigma).
\]
It is easy to observe that such a functional is continuous and coercive and that the minimization is performed on a compact set $K$: thus, the existence of (at least) a solution of \eqref{eq:regularization} is guaranteed by classic arguments (see e.g. \cite{benning2018modern}). Nevertheless, we cannot conclude the uniqueness of such solutions due to the nonlinearity of $\Phi$.

The solutions of \eqref{eq:regularization} are also stable with respect to perturbations of the datum $\Lambda^\delta$: introducing a sequence $\{\Lambda_k\} \subset \R^\nM$ such that $\Lambda_k \rightarrow \Lambda^\delta$, and considering a sequence of minimizers $\{\sigma_k\}$ of the functionals $\mathcal{J}_\lambda^k$, obtained by replacing $\Lambda^\delta$ with $\Lambda_k$ in $\mathcal{J}_\lambda^\delta$, then the limit of every convergent subsequence of $\{\sigma_k\}$ is a minimizer of $\mathcal{J}_\lambda^\delta$. This can be proved, with slight modifications, as in \cite[Theorem 10.2]{engl1996regularization}.

We now focus our attention on extending the convergence result \cite[Theorem 10.4]{engl1996regularization} to the optimization problem \eqref{eq:regularization}. Indeed we prove that, under suitable assumptions on $\Phi$ and $\sigma^\dag$ and for a specific choice of $\lambda$, as $\delta \rightarrow 0$, the minimizers $\sigmaL$ of \eqref{eq:regularization} converge to a solution of the inverse problem $\sigma^\dag$. 
We state the result in a general formulation, outlining all the properties required on the operator $\Phi$, which are verified in our setup as discussed in Remark \ref{rem:verify}. Hereafter, we denote by $J_{\Phi}(\sigma) \in \R^{\nM \times \nt}$ the Jacobian matrix of $\Phi$ computed at a point $\sigma \in \R^\nt$.

\begin{prop}
\label{prop:conv_rate}
    Let $\sigmaL$ be a local minimizers of \eqref{eq:regularization}, being $\mathcal{R}$ as in \eqref{eq:glob_reg} and $\Phi$ such that the following assumptions are satisfied:
    \begin{enumerate}
        \item Source condition: there exist $p^\dag \in \partial \mathcal{R}(\sigma^\dag)$ and $w \in \R^m$ such that
        \begin{equation}
        p^\dag = J_{\Phi}(\sigma^\dag)^T w
           \label{eq:SC}
        \end{equation}
        \item Mild non-linearity of $\Phi$ in $\sigma^\dag$: there exists $\gamma > 0$ such that
        \begin{equation}
            \| \Phi(\sigma) - \Phi(\sigma^\dag) - J_{\Phi}(\sigma^\dag)(\sigma - \sigma^\dag) \|^2 \leq \gamma \| \sigma - \sigma^\dag \|^2
            \qquad \forall \sigma \in K_{0,1}.
            \label{eq:nonlin_ass_dag}
        \end{equation}
    \end{enumerate}
    Assume moreover that $ 2 \gamma \| w\| \leq \rho$. Then, as $\lambda \rightarrow 0$, the sequence $\sigmaL$ converges to $\sigma^\dag$ and, for the choice $\lambda \sim \delta$, the following convergence rate holds:
    \begin{equation}
        \| \sigmaL -\sigma^\dag \| = O(\sqrt{\delta})
        \label{eq:conv_rate}
    \end{equation}
\end{prop}

\begin{proof}
    Since $\sigmaL$ is a minimizer of \eqref{eq:regularization}, it holds 
    \[
    \frac{1}{2} \| \Phi(\sigmaL) - \Ld \|^2  + \lambda \mathcal{R}(\sigmaL) \leq \frac{1}{2} \| \Phi(\sigma^\dag)- \Ld\|^2 + \lambda \mathcal{R}(\sigma^\dag); 
    \]
    subtracting the term $\lambda \langle p^\dag, \sigmaL - \sigma^\dag\rangle$ on both sides, being $p^\dag$ as in \eqref{eq:SC}, we get
    \begin{equation}
    \label{eq:aux1}
    \frac{1}{2} \| \Phi(\sigmaL) - \Ld \|^2 + \lambda \big(\mathcal{R}(\sigmaL) - \mathcal{R}(\sigma^\dag) - \langle p^\dag, \sigmaL - \sigma^\dag\rangle \big) \leq \frac{1}{2} \|\eta \|^2 - \lambda \langle p^\dag, \sigmaL - \sigma^\dag\rangle.
    \end{equation}
    The second term on the left-hand side of \eqref{eq:aux1} can be interpreted in terms of the Bregman divergence associated with the convex functional $\mathcal{R}$. In particular, it can be denoted as $D_{\mathcal{R}}^{p^\dag}(\sigmaL,\sigma^\dag)$, where for any $\sigma_1,\sigma_2$ the Bregman divergence from $\sigma_2$ to $\sigma_2$, with respect to $p_2 \in \partial\mathcal{R}(\sigma_2)$ is defined as
    \[
    D_{\mathcal{R}}^{p_2}(\sigma_1,\sigma_2) = \mathcal{R}(\sigma_1) - \mathcal{R}(\sigma_2) - \langle p_2, \sigma_1 - \sigma_2 \rangle.
    \]
    Notice that, for the specific choice of $\mathcal{R}$ outlined in \eqref{eq:glob_reg}, it holds that, for $\sigma_1,\sigma_2 \in K$, any $p_2 \in \partial \mathcal{R}$ can be written as $\rho \sigma_2 + q_2$, with $q_2 \in \partial R(\sigma_2)$, and thanks to the convexity of $R$,
    \[
    \begin{aligned}
      D_{\mathcal{R}}^{p_2}(\sigma_1,\sigma_2) &= R(\sigma_1) + \frac{\rho}{2} \|\sigma_1\|^2 - R(\sigma_2) - \frac{\rho}{2} \|\sigma_2\|^2 - \langle \rho \sigma_2 + q_2, \sigma_1-\sigma_2 \rangle \\
       & =\frac{\rho}{2} \left(\|\sigma_1\|^2 - \|\sigma_2\|^2 -2\langle \sigma_2, \sigma_1 -\sigma_2 \rangle \right) + \big( R(\sigma_1) - R(\sigma_2) - \langle q_2, \sigma_1-\sigma_2 \rangle \big) \\
       & \geq \frac{\rho}{2} \| \sigma_1 - \sigma_2 \|^2.
    \end{aligned}
    \]
    Inserting this in \eqref{eq:aux1}, and employing the source condition \eqref{eq:SC}, we obtain
    \begin{equation}
        \label{eq:aux12}
\frac{1}{2} \| \Phi(\sigmaL) - \Ld \|^2 + \frac{\lambda \rho}{2} \| \sigmaL - \sigma^\dag \|^2 \leq \frac{1}{2} \|\eta \|^2 - \lambda \langle w, J_\Phi(\sigma^\dag) (\sigmaL - \sigma^\dag)\rangle.
    \end{equation}
    Let us now focus on the second term on the right-hand side of \eqref{eq:aux12}: adding and subtracting various terms,
    \[
    - \langle w, J_\Phi(\sigma^\dag) (\sigmaL - \sigma^\dag)\rangle = \langle w, \big( \Phi(\sigmaL) - \Phi(\sigma^\dag) -J_\Phi(\sigma^\dag) (\sigmaL - \sigma^\dag) \big) + \big( \Ld - \Phi(\sigmaL) \big) + \big( \Phi(\sigma^\dag) - \Ld \big) \rangle;
    \]
    thus, inserting this in \eqref{eq:aux12},
    \[
    \begin{aligned}
        \frac{1}{2} \| \Phi(\sigmaL) - \Ld \|^2 
        + \lambda \langle w, \Phi(\sigmaL) - \Ld \rangle        
        + \frac{\lambda \rho}{2} \| \sigmaL - \sigma^\dag \|^2 \leq & \ \lambda \langle w, \Phi(\sigmaL) - \Phi(\sigma^\dag) -J_\Phi(\sigma^\dag) (\sigmaL - \sigma^\dag) \rangle \\
        & \ +\frac{1}{2} \|\eta \|^2 + \lambda \langle w, \eta \rangle ,
    \end{aligned}
    \]
    and as a consequence of \eqref{eq:nonlin_ass_dag}
    \[
    \frac{1}{2} \| \Phi(\sigmaL) - \Ld + \lambda w\|^2 + \frac{\lambda \rho}{2} \| \sigmaL - \sigma^\dag\|^2 \leq \lambda \gamma \| w\| \| \sigmaL - \sigma^\dag \|^2 + \| \eta + \lambda w\|^2.
    \]
    Finally, neglecting the first term on the left-hand side, and leveraging the inequality $2\gamma\| w \| \leq \rho$,
    \begin{equation}
        \label{eq:conv_rate_expl}
    \lambda \| \sigmaL - \sigma^\dag\|^2 \leq \frac{4}{\rho - 2 \gamma \| w\|} \big( \delta^2 + \lambda^2 \|w \|^2\big)
    \end{equation}
    and for the choice $\lambda \sim \delta$ we get that $\| \sigmaL - \sigma^\dag\|^2 = O(\delta)$, hence the thesis.
\end{proof}

\begin{rem} \label{rem:verify}
   The functional $\Phi$ of EIT satisfies the assumptions of Proposition \ref{prop:conv_rate}: indeed, the second statement of Theorem \ref{thm:Lechleiter} implies that for every $\sigma \in K_{0,1}$ (and in particular for $\sigma^\dag$) the Jacobian matrix $J_\Phi(\sigma)$ is injective, hence its transpose is surjective, and the source condition \eqref{eq:SC} is verified at any point $p^\dag$. 
    Condition \eqref{eq:nonlin_ass_dag} is moreover verified by $\Phi$ in view of the value inequality and of the Lipschitz continuity of the Jacobian, guaranteed again by Theorem \ref{thm:Lechleiter}. \\
\end{rem}

\section{Proximal-Gradient Method to solve \eqref{eq:regularization}}
\label{sec:PGM}

We now consider a first-order iterative method to solve the minimization problem \eqref{eq:regularization}. To simplify its formulation, we introduce the following notation:
\begin{equation}
    \label{eq:Jfg}
\begin{aligned}
    J_\lambda^\delta(\sigma) = f(\sigma) + \lambda g(\sigma), \qquad f(\sigma) = \frac{1}{2}\| \Phi(\sigma) - \Ld \|^2 + \frac{\lambda \rho}{2} \|\sigma \|^2, \qquad g(\sigma) = R(\sigma) + \upchi_K(\sigma).
\end{aligned}
\end{equation}

In this setup, $f$ is a smooth functional, meaning that it is differentiable with Lipschitz continuous gradient, whereas $g$ is proper, convex, and continuous, but is in general non-differentiable. Therefore, a natural choice to approximate $\sigmaL$ is to rely on a proximal-gradient, or forward-backward, scheme.
This accounts for constructing a sequence $\{\sigma^{(n)}\}$ as follows:
\begin{equation}
    \sigma^{(n+1)} = \prox_{\mu \lambda g}\big(\sigma^{(n)} - \mu \nabla f(\sigma^{(n)})\big),
    \label{eq:PG}
\end{equation}
where $\mu > 0$ is a constant step-size. The convergence properties of such iterates, under suitable limitations on $\mu$ and various assumptions of $f,g$,  are widely studied in the literature. For convex $g$ and smooth $f$, provided that $\mu$ is smaller than the inverse of the Lipschitz constant of $\nabla f$, \cite[Theorem 1.3]{Beck_Teboulle_2009} proves that the cluster point of the iterates of \eqref{eq:PG} satisfies the (necessary) optimality conditions associated with $\mathcal{J}_\lambda^\delta$, also providing a convergence rate of an involved quantity. 

In order to prove the convergence of the sequence of the iterates (to a minimizer of $\mathcal{J}_\lambda^\delta$) it is nevertheless necessary to require the convexity of $f$ (see, e.g., \cite[Theorem 1.2]{Beck_Teboulle_2009} or \cite[Section 28.5]{bauschke2017correction}). Unfortunately, in our case this property is only guaranteed for large values of $\rho$: in particular, due to the nonlinearity of $\Phi$, the term $\frac{1}{2}\| \Phi(\sigma) - \Ld\|^2$ is not convex, even though it can be proved to be weakly convex. For the sake of completeness, we briefly discuss two results related to the convexity of similar functionals in the context of EIT. The first one, \cite[Theorem 4.9]{Kindermann2021OnTT}, shows that, in the continuum model of EIT, the squared-norm mismatch functional is convex when restricted to a suitable subset of $\R^\nM$, but cannot guarantee that the iterates generated by an iterative scheme analogous to \eqref{eq:PG} belong to such a set. The second one, \cite[Lemma 4.7]{harrach2023calderon}, proves that, in the continuum model and in the presence of finitely many measurements, the forward operator $F_m$ (substituting $\Phi$ in $f$) is convex with respect to the (semidefinite) Loewner of semidefinite positive operators. Unfortunately, this does not translate into the convexity of the combination of such a functional with the squared Frobenius norm, which would eventually lead to a formulation analogous to our quadratic mismatch term; as an alternative, \cite{harrach2023calderon} relies on an algorithm based on the minimization of a linear functional on a suitable convex set.

An alternative approach to the study of the convergence of first-order methods relies on Kurdyka-Łojasiewicz (KL) conditions, which are less restrictive than convexity: see, for example, \cite{attouch2013convergence}. The results in \cite[Section 3.2]{hurault2023relaxed} provide the convergence rates for the function values and the difference of the iterates under weaker assumptions with respect to the ones we consider here (smoothness of $f$, weak-convexity of $g$), but the convergence of the iterates (to a stationary point of $\mathcal{J}_\lambda^\delta$, due to the lack of convexity) is ensured only if the KL condition is satisfied, see also \cite[Theorem 5.1]{attouch2013convergence}. 
The verification of the KL conditions in the case of EIT is an open problem, and the local injectivity property analyzed in the current dissertation does not prove to be useful: by contrast, the surjectivity of the Jacobian $J_\Phi$ would be beneficial (see \cite[Theorem 3.2]{li2018calculus}).

We prove a different result on the sequence $\sigma^{(n)}$, which takes advantage of the local injectivity of $\Phi$. This is an adaptation of Theorem 2 in \cite{blumensath2013compressed}, which is not formulated for a proximal-gradient scheme but for Iterative Hard Thresholding (essentially equivalent to a projected gradient algorithm).

We start by substituting the definition of $f$ in \eqref{eq:PG}, obtaining the following expression of the iterates of PGM:
\begin{equation}
    \sigma^{(n+1)} = \prox_{\mu \lambda g}\big(\sigma^{(n)} - \mu \lambda \rho \sigma^{(n)} + \mu J_\Phi(\sigma^{(n)})^T(\Ld - \Phi(\sigma^{(n)}))\big),
    \label{eq:PG-Phi}
\end{equation}
with $\mu >0$, where $J_\Phi(\sigma)^T \in \R^{\nt \times \nM}$ denotes the transpose of the Jacobian matrix of $\Phi$ at $\sigma \in \R^\nt$.
The following result characterizes the cluster point $\overline{\sigma}$ of the iterates produced by \eqref{eq:PG-Phi}. In particular, it provides a bound for the distance between $\overline{\sigma}$ and any $\sigma \in K$ only in terms of the value of the functional $\mathcal{J}_\lambda$ evaluated in such a $\sigma$. This result can be effectively employed to discuss the asymptotic properties of $\overline{\sigma}$ as $\delta \rightarrow 0$, as discussed in Corollary \ref{cor:1}, but may also help in motivating the selection of the space $K$, as observed in Corollary \ref{cor:2} and further explored in Section \ref{sec:sec6}. Notice that we do not need to specify the expression of the functional $g$, which will be the object of Section \ref{sec:sec5}.

\begin{theorem}
\label{thm:blumensath}
Let $J_{\lambda}^\delta(\sigma) = f(\sigma) + \lambda g(\sigma)$, being $f$ as in \eqref{eq:Jfg} and $g$ any non-negative, coercive, continuous, and convex functional, whose domain is contained in $K_{0,1}$. Assume that $\Phi$ satisfies:
    \begin{enumerate}
        \item Restricted Isometry Property: there exists $0<\alpha \leq \beta$ such that
        \begin{equation}
            \alpha \| \sigma_1 - \sigma_2 \|^2 \leq \| J_\Phi(\sigma) (\sigma_1 - \sigma_2) \|^2 \leq  \beta \| \sigma_1 - \sigma_2 \|^2 \quad \forall \sigma \in K_{0,1}, \quad \forall \sigma_1,\sigma_2 \in \R^\nt
            \label{eq:RIP}
        \end{equation}
        \item Mild non-linearity: there exists $\gamma > 0$ such that
        \begin{equation}
            \| \Phi(\sigma_1) - \Phi(\sigma_2) - J_{\Phi}(\sigma_2)(\sigma_1 - \sigma_2) \|^2 \leq \gamma \| \sigma_1 - \sigma_2 \|^2 \qquad \forall \sigma_1,\sigma_2 \in K_{0,1}.
            \label{eq:nonlin_ass}
        \end{equation}
    \end{enumerate}
    Require moreover that $\alpha, \beta, \gamma, \lambda,\rho$ and $\mu$ satisfy 
    \begin{equation}
    \mu \leq \frac{1}{2\beta}, \qquad \mu \leq \frac{1}{2 \lambda \rho}, \qquad 0 < \alpha + \lambda \rho -2\gamma.
        \label{eq:parameters}
    \end{equation} 
    Then, the sequence $\sigma^{(n)}$ defined in \eqref{eq:PG-Phi} converges to a cluster point $\overline{\sigma}$ such that, for any $\sigma \in \operatorname{dom}(g)$
    \begin{equation}
    \| \sigma - \overline{\sigma} \|^2 \leq \frac{4}{\alpha + \lambda\rho - 2\gamma} \mathcal{J}_\lambda(\sigma).
        \label{eq:convergence}
    \end{equation} 
\end{theorem}
The proof of Theorem \ref{thm:blumensath} is provided in Appendix A.

Let us here briefly discuss the introduced assumptions.
\begin{enumerate}
    \item The hypotheses on the functional $f$ are verified when the operator $\Phi$ is the one of EIT, for a sufficiently large number of electrodes $p$ and small discretization size $h$. Indeed, the local injectivity  of $F'_{p,h}$ and the Lipschitz continuity and differentiability of $F_{p,h}$ on $\Vt$ in Theorem \ref{thm:Lechleiter} ensure the lower and upper bounds in \eqref{eq:RIP}, whereas \eqref{eq:nonlin_ass} is guaranteed by the Lipschitz continuity of $F'_{p,h}$.
    \item The requirement $0 < \alpha + \lambda \rho -2\gamma$ in \eqref{eq:parameters} imposes an important connection between the ill-conditioning of $J_\Phi$ (expressed by $\alpha$, as $\| J_\Phi(\sigma)\|_2^2 \ge \alpha$ from \eqref{eq:RIP}) and the non-linearity of $\Phi$ (encoded in $\gamma$). It is in general very difficult to check if the stronger condition $2 \gamma < \alpha$ is verified in the case of EIT; nevertheless, we can leverage the presence of the regularization parameters $\lambda \rho$. Despite this being restrictive in an asymptotic scenario as $\lambda \rightarrow 0$, it is possible to verify that the requirements in \eqref{eq:parameters} are less restrictive than the ones entailed by the original computations contained in \cite{blumensath2013compressed}, which would imply a condition like $\beta + \lambda \rho < \frac{3}{2}(\alpha + \lambda \rho) - 4\gamma$.
\end{enumerate}

As a corollary of Proposition \ref{prop:conv_rate} and of Theorem \ref{thm:blumensath}, we deduce the following convergence rate for the error between the cluster point $\overline{\sigma}$ of \eqref{eq:PG-Phi} and the solution of the inverse problem.
\begin{cor} \label{cor:1}
Let $\overline{\sigma}$ be the cluster point of the iterates \eqref{eq:PG-Phi} and $\sigma^\dag$ the solution of the inverse problem \eqref{eq:cs}. Suppose the assumptions of Proposition \ref{prop:conv_rate} and of Theorem \ref{thm:blumensath} are verified, and that moreover the last expression in \eqref{eq:parameters} is replaced by $2\gamma < \alpha$. Then, for sufficiently small $\delta$, under the choice $\lambda = C\delta$, there exists a constant $c$ (depending on $\alpha,\gamma,\rho, \sigma^\dag$ and $C$) such that
\begin{equation}
\| \overline{\sigma} - \sigma^\dag\| \leq c \sqrt{\delta}
    \label{eq:overall_rate}
\end{equation}
\end{cor}
\begin{proof}  
Substituting $\sigmaL$ in \eqref{eq:convergence} and leveraging the explicit expression of the convergence rate \eqref{eq:conv_rate_expl},
\[
\begin{aligned}
         \| \overline{\sigma}  - \sigma^\dag \|^2 &\leq 2\| \overline{\sigma} - \sigmaL \|^2 + 2 \| \sigmaL - \sigma^\dag \|^2 \\
         & \leq \frac{8}{\alpha + \lambda \rho - 2\gamma} \mathcal{J}_\lambda^\delta(\sigmaL) + \frac{8}{\lambda(\rho - 2\gamma \| w\|)}(\delta^2 + \lambda^2 \| w\|^2).
\end{aligned}
\]
Now, since by definition $\mathcal{J}_\lambda^\delta(\sigmaL) \leq \mathcal{J}_\lambda^\delta(\sigma^\dag)$,
\[
\begin{aligned}
         \| \overline{\sigma}  - \sigma^\dag \|^2
         & \leq \frac{8}{\alpha + \lambda \rho - 2\gamma} \left(\frac{\delta^2}{2} + \frac{\lambda \rho}{2}\| \sigma^\dag\|^2 + \lambda g(\sigma^\dag) \right) + \frac{8}{\lambda (\rho - 2\gamma \| w\|)}(\delta^2 + \lambda^2 \| w\|^2).
\end{aligned}
\]
Exploiting now the fact that $\alpha - 2\gamma>0$ (and $\lambda \rho >0$), under the choice $\lambda = C \delta$, we recover the desired estimate for $\delta<1$:
\begin{equation}
\| \overline{\sigma} - \sigma^\dag\|^2 \leq \left(\frac{4 + 4C\rho + 8 g(\sigma^\dag)}{\alpha - 2 \gamma} + \frac{8(1 + C^2\| w \|^2)}{C(\rho - 2\gamma \| w\|)}\right) \delta
    \label{eq:overall_rate_expl}
\end{equation}
\end{proof}

The next corollary is instead of prominent importance when considering the choice of the regularization functional $g$. Let us focus on the model expressed in \eqref{eq:Jfg}, assuming that $g$ is the sum of a non-negative, convex, and continuous functional $R$ (whose domain is, without loss of generality, $\R^\nt$), and the characteristic function $\upchi_K$ of the compact, convex set $K\subset K_{0,1}$. In particular, the following result discussed the importance of an educated choice of the set $K$, determining the domain of $g$. 

\begin{cor} \label{cor:2}
Suppose the assumptions of Theorem \ref{thm:blumensath} are verified, and that the last expression in \eqref{eq:parameters} is replaced by $2\gamma < \alpha$. Let $\sigma_K$ be the orthogonal projection of the exact solution $\sigma^\dag$ onto $K$, i.e.,
\begin{equation}
\sigma_K = \operatorname{proj}_K(\sigma^\dag) = \argmin_{\sigma \in K}\big\{\| \sigma^\dag - \sigma\|\big\}.
    \label{eq:sigmaK}
\end{equation}
Then, for sufficiently small $\delta$, under the choice $\lambda = C\delta$, there exist two constants $c_1,c_2$ (depending on $\alpha,\gamma,\rho,g$ and $C$) such that
\begin{equation}
\| \overline{\sigma} - \sigma^\dag\| \leq c_1 \sqrt{\delta} + c_2 \| \sigma_K -\sigma^\dag\|.
    \label{eq:motivation_oracle}
\end{equation}
\end{cor}
\begin{proof}
    We employ the inequality \eqref{eq:convergence} replacing $\sigma$ by $\sigma_K$. Then, it holds
\[
\begin{aligned}
\| \overline{\sigma} - \sigma^\dag\|^2 &\leq 2 \| \overline{\sigma} - \sigma_K \|^2  + 2 \| \sigma_K - \sigma^\dag\|^2 \\
& \leq \frac{8}{\alpha + \lambda \rho - 2\gamma} \mathcal{J}_{\lambda}^\delta(\sigma_K) + 2 \| \sigma_K - \sigma^\dag \|^2 \\
& \leq \frac{8}{\alpha + \lambda \rho - 2\gamma} \left( \delta^2 + \| \Phi(\sigma_K) - \Phi(\sigma^\dag) \|^2 + \frac{\lambda \rho}{2}\| \sigma_K\|^2 + \lambda g(\sigma_K) \right) + 2 \| \sigma_K - \sigma^\dag \|^2
\end{aligned}
\]
We now recall that the map $\Phi$ is Lipschitz-continuous with a constant smaller than $\beta$, due to \eqref{eq:RIP} and to the mean-value inequality. Moreover, since $K$ is contained within the compact set $K_{0,1}$, the quantities $\| \sigma_K \|^2$ and $g(\sigma_K) = R(\sigma_K)$ can be bounded by the constants $\Gamma_1$ and $\Gamma_2$ respectively, both independent of $K$. Thus, for $\lambda =C \delta$ and $\delta<1$, also employing that $\alpha - 2\gamma>0$, we get
\begin{equation}
\| \overline{\sigma} - \sigma^\dag\|^2 \leq \frac{8 + 4 C\rho \Gamma_1 + \Gamma_2}{\alpha - 2 \gamma} \ \delta + \left( \frac{8 \beta}{\alpha - 2 \gamma} +2 \right)\| \sigma_K - \sigma^\dag\|^2.
    \label{eq:motivation_oracle_expl}
\end{equation}
\end{proof}

When comparing \eqref{eq:overall_rate} and \eqref{eq:motivation_oracle}, one observes that, although the latter bound does not guarantee the asymptotic accuracy of $\overline{\sigma}$ as $\delta \rightarrow 0$, it clearly outlines its dependence on the choice of $K$. In particular, for a fixed $\delta$, \eqref{eq:motivation_oracle} shows that the error might be reduced if $K$ is chosen such that $\|\sigma_K -\sigma^\dag\|$ is extremely small, or even $0$. \\
If the solution $\sigma^\dag$ is known to be $s-$sparse (i.e., to have at most $s$ non-vanishing components), one possible way to leverage this information (see e.g. \cite{blumensath2013compressed}) is by considering $K$ as the union of all the $s-$dimensional coordinate hyperplanes in $\R^\nt$. The projection onto such a space can be easily computed (by selecting the $s$ largest components of a vector, as in the Iterative Hard Thresholding algorithm in \cite{blumensath2013compressed}), but an overestimation of the sparsity level $s$ may still lead to an inefficient reconstruction. \\
In our case, we wish to select $K$ in a more insightful way. In particular, suppose that an \textit{Oracle} function is available, taking as an input the measurements $\Lambda^\delta$ and returning the exact support of $\sigma^\dag$. Employing such an output as the set $K$ in algorithm \eqref{eq:PG-Phi} would allow canceling the last term on the right-hand side of \eqref{eq:motivation_oracle}, thus obtaining better estimates. The quest for such an Oracle functional can also be interpreted as the estimation of the \textit{support}, or the \textit{sparsity pattern} of the solution, a problem that has been extensively studied in the context of compressed sensing (see, e.g., \cite{fletcher2009necessary,meinshausen2006high,wainwright2006sharp}). Our approach, discussed in Section \ref{sec:sec6}, relies instead on statistical learning techniques, and in particular based on Graph Neural Networks, to provide a data-driven approximation of the optimal Oracle functional.

\section{Some explicit regularization models}
\label{sec:sec5}
We now want to study more closely some choices of $R$ and $K$ in \eqref{eq:regularization}, which will also reflect in more explicit expressions for the proximal operator of $g$ appearing in \eqref{eq:PG-Phi} and defined in \eqref{eq:Jfg}.

\par
In particular, we consider the following possible expressions for $R$:
\begin{enumerate}[(1)]  
    \item Sparsity-promoting regularization through $\ell^1$ norm:
    \begin{equation}
    R(\sigma) = \| \sigma - \sigma_0 \|_1, \qquad \sigma_0 \in [c_0,c_1] \subset \R^\nt.
        \label{eq:R_ell1}
    \end{equation}
    This reflects the assumption that $\sigma$ differs from a known reference conductivity $\sigma_0$ only in a few components. The proximal map in this case reads as follows:
    \[    \sigma^*=\prox_{\lambda \mu R}(\sigma) = \sigma_0 + S_{\lambda \mu}(\sigma-\sigma_0),
    \]
    where $S_t$ is the element-wise soft-thresholding function, namely $S_t\colon \R^\nt \rightarrow \R^\nt$ such that, for $i = 1,\ldots,\nt$,
    \[
    [S_t(\nu)]_i = \operatorname{sign}(\nu_i) \max(0,|\nu_i|-t).
    \]
    
    \item Anisotropic Total Variation (TV) on meshes (see \cite{colibazzi2022learning}):
    \begin{equation} \label{eq:R_TV}
        R(\sigma) = \operatorname{TV}(\sigma) = \sum_{i=1}^n \sum_{k \in \mathcal{N}_i} w_{ik} |\sigma_i - \sigma_k|, 
    \end{equation}
    where, for each $i$, $\mathcal{N}_i$ denotes the set of indices $k \in \{1,\ldots,\mbox{deg}(i)\}$, with  $\mbox{deg}(i)$ denoting the valence of vertex $i$th. 
    The positive weight $w_{ik}$ is instead defined as the inverse of the (Euclidean) distance between the $i$-th vertex and its $k$-th adjacent vertex.
     The proximal map associated with this choice     \begin{equation} \sigma^*= \prox_{\lambda \mu R}(\sigma)=
\arg\min_{x \in \R^{n}} \left\{ \frac{1}{2}\|x-\sigma\|_2^2 +\lambda \mu TV(x) \right\}
\label{eq:pgmprox_TV_1}
\end{equation}
can be approximately computed according to a procedure proposed in \cite{Osher2009} and generalized in \cite{colibazzi2022learning} for polygonal meshes. Let $ \sigma^*$ be  the fixed point of the following equations, stated on its components
     \begin{equation}
    \label{eq:prox_TV}
        \sigma^*_i = \argmin_{x \in \R} \left\{ \frac{1}{2} (x-\sigma_i)^2 + \lambda \mu \sum_{k \in \mathcal{N}_i} w_{ik} |x-\sigma^*_k|\right\}.
    \end{equation}

    Following \cite[Theorem 3.2 and Remark 3.1]{Osher2009}, the solution of problem \eqref{eq:prox_TV} can be obtained as the limit of the sub-iterations
    \begin{equation}
    \label{eq:sub_it_TV}
        \sigma^{(l+1)}_i = \argmin_{x \in \R} \left\{ \frac{1}{2} (x-\sigma_i)^2 + \lambda \mu \sum_{k \in \mathcal{N}_i} w_{ik} |x-\sigma^{(l)}_k|\right\},
    \end{equation}
    and for each $i$ and $l$ the unique solution of \eqref{eq:sub_it_TV} can be efficiently computed by the median formula proposed in \cite{Osher2009}
\begin{equation}
\label{eq:closed}
        \sigma^{(l+1)}_i =   
       \mbox{median}\left\{\sigma^{(l)}_1, \ldots, \sigma^{(l)}_n, \sigma_i^{(l)}+ \lambda \mu W_0,
       \sigma_i^{(l)}+ \lambda \mu W_1\ldots, \sigma_i^{(l)}+ \lambda \mu W_n \right\},
\end{equation}
where the values ${\sigma_k}$ of the vertices $k \in \mathcal{N}_i$ are sorted in increasing order, and  
\begin{equation}
    W_j = - \sum_{k=1}^{j} w_k + \sum_{k=j+1}^{n} w_k, \quad \quad j=0,..,n.
\label{eq:Wk}
\end{equation}
A discussion on the convergence to the global anisotropic TV problem \eqref{eq:pgmprox_TV_1} by iterating the local optimization problem \eqref{eq:prox_TV}, is provided in \cite{Osher2009}. 
\end{enumerate}

Regarding $K$ we consider the following cases:
\begin{enumerate}[(i)]
    \item Box constraint in $\R^\nt$: given $0 <c_0 <c_1$,
    \begin{equation}
        K = [c_0,c_1]^\nt, \qquad \upchi_K(\sigma) = \prod_{i=1}^n \upchi_{[c_0,c_1](\sigma_i)}.
        \label{eq:K_cube}
    \end{equation}
    This restriction is mandatory in order to ensure that the conductivity $\sigma$ is non-vanishing and bounded away from $0$, hence the EIT problem is well-defined.
    \item  \textbf{Oracle}-based projection:
    \begin{equation}
        K = [c_0,c_1]^\nt \cap \Pi_\mathcal{O},
        \label{eq:K_oracle}
    \end{equation}
    being $\Pi_\mathcal{O}$ a coordinate hyperplane conducted through a reference point $\sigma_0$, i.e., of the form:
    \begin{equation}
        \Pi_\mathcal{O} = \big\{ x \in \R^\nt: \quad x_i - \sigma_{0,i} = 0 \ \ \forall i \in \{1,\ldots,n\} \setminus I_{\mathcal{O}}\big\},
    \end{equation}
    where $I_{\mathcal{O}}$ denotes the set of active coordinates. The knowledge of $\Pi_\mathcal{O}$ (i.e., of the expected support of $\sigma - \sigma_0$) can be seen as the outcome of a separate support estimation problem, and acts as an Oracle for the PGM scheme \eqref{eq:PG-Phi}. The projection onto $\Pi_\mathcal{O}$ can be performed as follows:
    \[
    \operatorname{proj}_{\Pi_\mathcal{O}}(\sigma) = \sigma_0 + M_\mathcal{O} \odot (\sigma - \sigma_0),
    \]
    where $\odot$ denotes the element-wise product of vectors and $M_{\mathcal{O}} \in \R^{\nt}$ is the \textit{mask} associated with the Oracle $\mathcal{O}$, namely, a vector such that $[M_\mathcal{O}]_i = 1$ if $i \in I_\mathcal{O}$ and $0$ otherwise.
\end{enumerate}

The following result significantly simplifies the computation of $\prox_{\lambda \mu g}$ in \eqref{eq:PG-Phi} for our proposed choices of $R$ and $K$.

\begin{prop}
    \label{prop:summative_prox}
\textbf{[Proximal map of $g$]} 
Let $g$ be a proper, convex and continuous function defined as 
\begin{equation}
g(\sigma) = R(\sigma) + \upchi_K(\sigma),
\label{eq:g_fun}
\end{equation}
where $R$ is a non-negative, proper, continuous, convex functional, and $K \subset K_{0,1} = [c_0,c_1]^n$ is a convex and closed set in $\R^n$.
If $R$ is chosen as in \eqref{eq:R_ell1} or \eqref{eq:R_TV} and $K$ as in \eqref{eq:K_cube} or \eqref{eq:K_oracle}, then the proximal of $g$ is \textit{summative}, 
$\prox_g = \prox_{\upchi_K} \circ \prox_R $,  and, in particular, it satisfies
\begin{equation}
\prox_{\lambda \mu g} = \operatorname{proj}_K \circ \ \prox_{\lambda \mu R} \ .
    \label{eq:summative_prox}
\end{equation}

\end{prop}
\begin{proof}
    Both the choices of $K$ lead to characteristic functions $\upchi_K(\sigma)$ which can be written as the product of $n$ functions, each in one component $\sigma_i$. This is also the case for $R$ in \eqref{eq:R_ell1}. Hence, in these cases, \eqref{eq:summative_prox} follows by \cite[Proposition 2.2]{chaux2009nested}. Instead, when $R$ is chosen as in \eqref{eq:R_TV}, both $R$ and $K$ fit the hypotheses of \cite[Proposition II.2]{pustelnik2017proximity}, which entails \eqref{eq:summative_prox}.
\end{proof}

The previous result can be extended also to other separable functionals $R$, as the one in \eqref{eq:R_ell1}, which can be decomposed with respect to the components of $\sigma$, i.e., that can be written as $R(\sigma) = \sum_{i = 1}^{\nt} \psi_i(\sigma_i)$, being $\psi_i\colon \R \rightarrow \R$ non-negative, continuous and convex. This class includes, for example, the choices $R(\sigma) = \| \sigma\|_p^p$ for $p\geq 1$. Moreover, the theoretical treatment of case \eqref{eq:R_TV} can be extended to functionals that have different definitions of the weights $w_{ik}$, also replacing the absolute value $|\sigma_i - \sigma_k|$ by the difference $\sigma_i-\sigma_k$. This includes, for example, wavelet operators defined on meshes. Moreover, it can be extended to the case $R(\sigma) = \operatorname{TV}(\sigma-\sigma_0)$, for $\sigma_0 \in [c_0,c_1]^n$.

In conclusion, we consider the solution of the variational model \eqref{eq:regularization} with $g(\sigma)$ defined  as in  \eqref{eq:g_fun}, which reads as
\begin{equation}
    \sigmaL \in \argmin_{\sigma \in \R^\nt} \left\{ \mathcal{J}_\lambda^\delta(\sigma):= \frac{1}{2}\| \Phi(\sigma) - \Ld\|^2 + \frac{\lambda \rho}{2} \| \sigma \|^2 + \lambda R(\sigma) + \upchi_K(\sigma)
\right\}.
\label{eq:regularization_1}
\end{equation}
We propose the following four versions of the Proximal-Gradient method defined in \eqref{eq:PG-Phi}, where the $\prox_{\lambda \mu g}$ operator is characterized by the four different combinations of $R$ and $K$ in \eqref{eq:regularization_1}:
\begin{enumerate}[leftmargin=1cm]
    \item [PGM-$\ell^1$]: $R$ is chosen as in \eqref{eq:R_ell1} and $K$ as in \eqref{eq:K_cube}, promoting sparsity of the difference from the reference $\sigma_0$ and imposing the state constraints on $\sigma$;
    \item [PGM-TV]: $R$ is chosen as in \eqref{eq:R_TV} and $K$ as in \eqref{eq:K_cube}, promoting sparsity of gradient of $\sigma$ and imposing the state constraints on $\sigma$;
    \item [PGM-$\ell^1$-$M_\mathcal{O}$]: $R$ is chosen as in \eqref{eq:R_ell1} and $K$ as in \eqref{eq:K_oracle}, employing a pre-trained Oracle to select the support of $\sigma - \sigma_0$;
    \item [PGM-TV-$M_\mathcal{O}$]: $R$ is chosen as in \eqref{eq:R_TV}, and $K$ as in \eqref{eq:K_oracle}.
\end{enumerate}

\section{Support Estimation via Oracle-Net}
\label{sec:sec6}
We devised an Oracle-based strategy to predict the expected support $I_{\mathcal{O}}$ of a conductivity distribution $\sigma$ starting from a set of measurements $\Lambda^{\delta}$.
The support $I_{\mathcal{O}}$ allows to classify $\sigma$ on the vertices of the mesh domain  $\Tau_h$ as belonging to an inclusion, when  $\sigma_i \neq \sigma_{0,i}$, or belonging to the background $\sigma_0$.  
The Oracle $\mathcal{O}$ produces a binary mask $M_{\mathcal{O}}$ on the vertices of the mesh $\Tau_h$, with $[M_{\mathcal{O}}]_i=1$ for vertex $v_i$ if $i \in I_{\mathcal{O}}$, and $0$ otherwise. 

We modeled  $\mathcal{O}$ with a Graph-U-Net, named Oracle-Net, which is a U-Net-like architecture adopted for graph/mesh data as in \cite{2019pytorch,Col2023DeepplugandplayPG}. 
The Graph-U-Net is an autoencoder architecture based on convolutional graph operators and gPool and gUnpool operators. The pool (gPool) operator samples some nodes to form a coarser mesh while the unpooling (gUnpool) operator performs the inverse process, by increasing the number of nodes exploiting the list of node locations selected in the corresponding gPool. 

 The architecture of the Oracle-Net neural network is depicted in Fig. \ref{fig:ONET}.
Each layer $T_\ell$ is characterized by the composition of a GCN-based graph convolution \cite{K7}, a  ReLU activation function $s$, and a gPool/gUnpool operator, here denoted by a generic $p$, and is applied to the $n_c$-dimensional input feature array $X \in \R^{n \times n_c}$, namely
\begin{equation}
\label{eq:layer}
T_\ell: X \mapsto s(GCN(p(X);\Theta_{\ell})),
\end{equation}
where $\Theta_{\ell}  \in \R^{n_c \times n_f}$ denotes the trainable weight matrix of layer $\ell$, and $n_f$ is the number of output features of the layer $\ell$. 
The Oracle-Net with $2P$ layers performs the following composite function 
\begin{equation}
\mathcal{O}(X) := \underbrace{T_{2P}^{gU} \circ \cdots \circ T_{P+1}^{gU}}_{decoder} \circ \underbrace{T_P^{gP} \circ \cdots \circ T_1^{gP}}_{encoder}(X),
\label{eq:Gcomp}
\end{equation}
where 
$T_\ell^{gP}$, for $\ell=1,\ldots,P$ applies a gPool operation to the feature vector, while $T_\ell^{gU}$, 
 for $\ell=P+1,\ldots,2P$ applies a gUnpool operation.  
The last decoder layer uses a sigmoid activation function $s$ in \eqref{eq:layer} which returns the probability that $v$ belongs to the support.
To finally produce the binary mask $M_{\mathcal{O}}$ in $\{0,1\}^n$ a thresholding is applied with value $\sigma_{th}$ in the inference phase.

The Oracle-Net takes as input a mesh $\Tau_h$ (hereon, we consider the same mesh for the discretization of the variable $\sigma$ and for the numerical solution of the differential problem, thus $\Tau=\Tau_h$) with its adjacency matrix, and a vector of weights $w \in \R^{m}$ for each vertex $v \in \Tau_h$ which moderates the influence of a given potential measurement on the point $v$. 
According to the selected measurement protocol the weight distribution $w$ changes, as illustrated in Fig.\ref{fig:ONETbis}(a) for adjacent injection, adjacent measurement,
$(I_{[E_k,E_{k+1}]},V_{[E_j,E_{j+1}]})$, and in Fig.\ref{fig:ONETbis}(b)
for opposite injection, adjacent measurement,  $(I_{[E_k,E_{k+p/2}]},V_{[E_j,E_{j+1}]})$.
In particular, each vector component $w_i$ associated with a given pair (injection, measurement)
captures the influence of the measurement on the vertex $v_i$ of the mesh and is computed as follows
\begin{equation}
w_i=\frac{1}{d_V^i+d_I^i} \, V_{[E_j,E_{j+1}]} \, ,
\end{equation}
where $d_V^i$ is the distance from the vertex $v_i$ to the midpoint between the pairs of electrodes involved in the measurements $[E_j,E_{j+1}]$, and $d_I^i$ is the distance between the vertex $v_i$ and the midpoint between the pairs of electrodes used for the injection, $[E_k,E_{k+1}]$ or $[E_k,E_{k+p/2}]$. 

\begin{figure}[h]
    \centering
    \begin{tabular}{c}       \fbox{\includegraphics[width=0.85\textwidth]{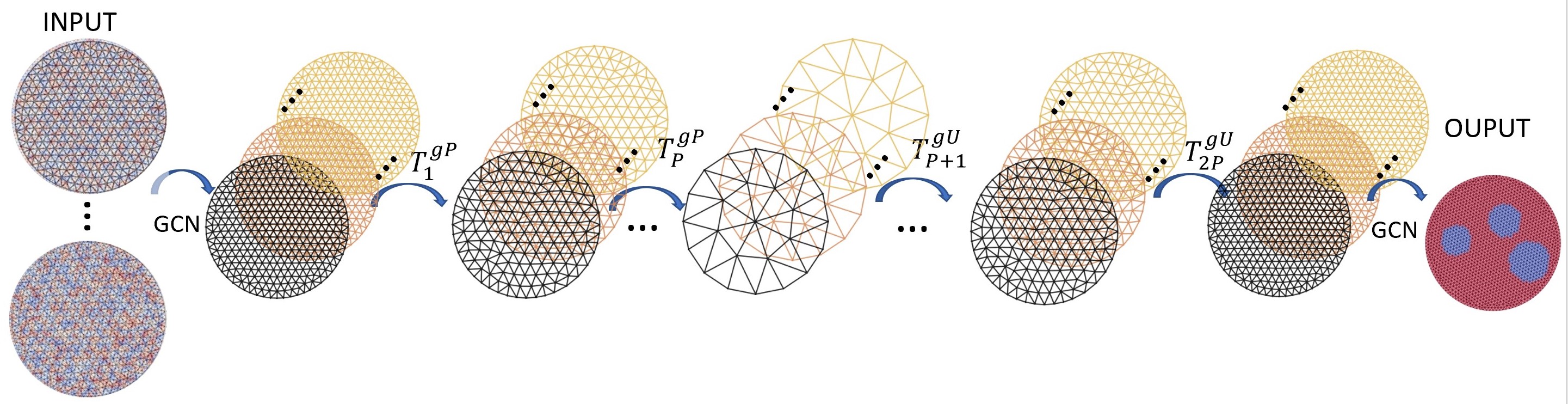}}\\
        \end{tabular}
    \caption{Oracle-Net architecture}
    \label{fig:ONET}
\end{figure}

\begin{figure}[h]
    \centering
    \begin{tabular}{cc}       
        \includegraphics[width=0.25\textwidth]{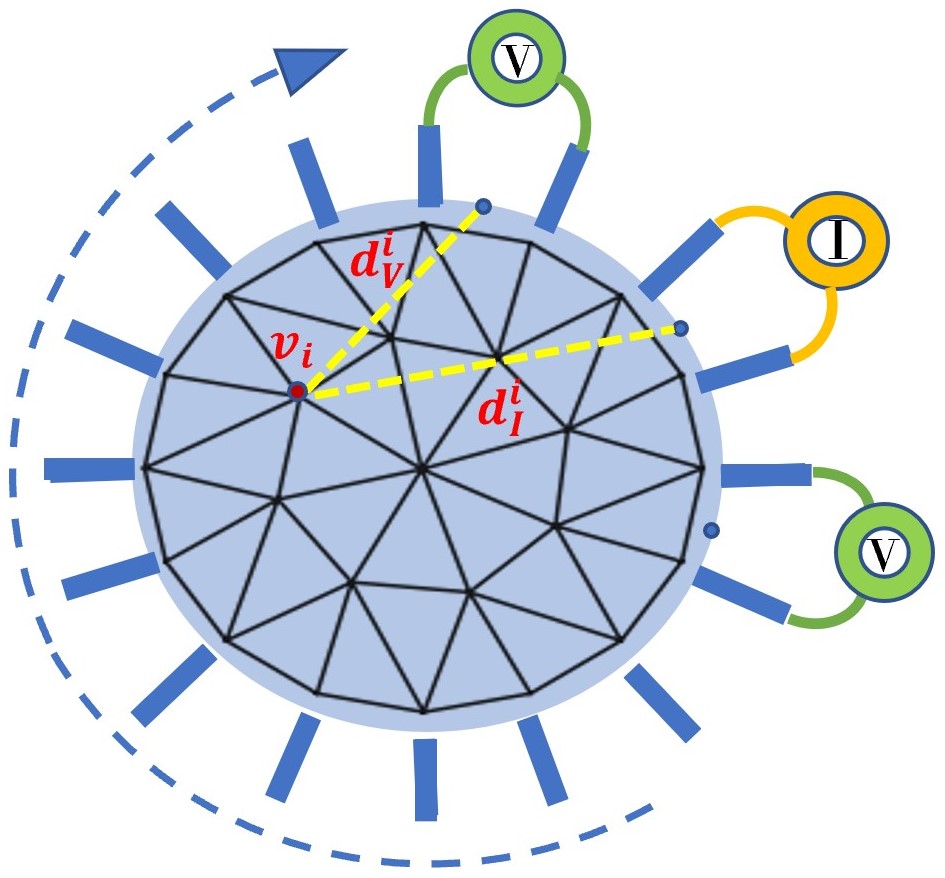} &
        \includegraphics[width=0.25\textwidth]{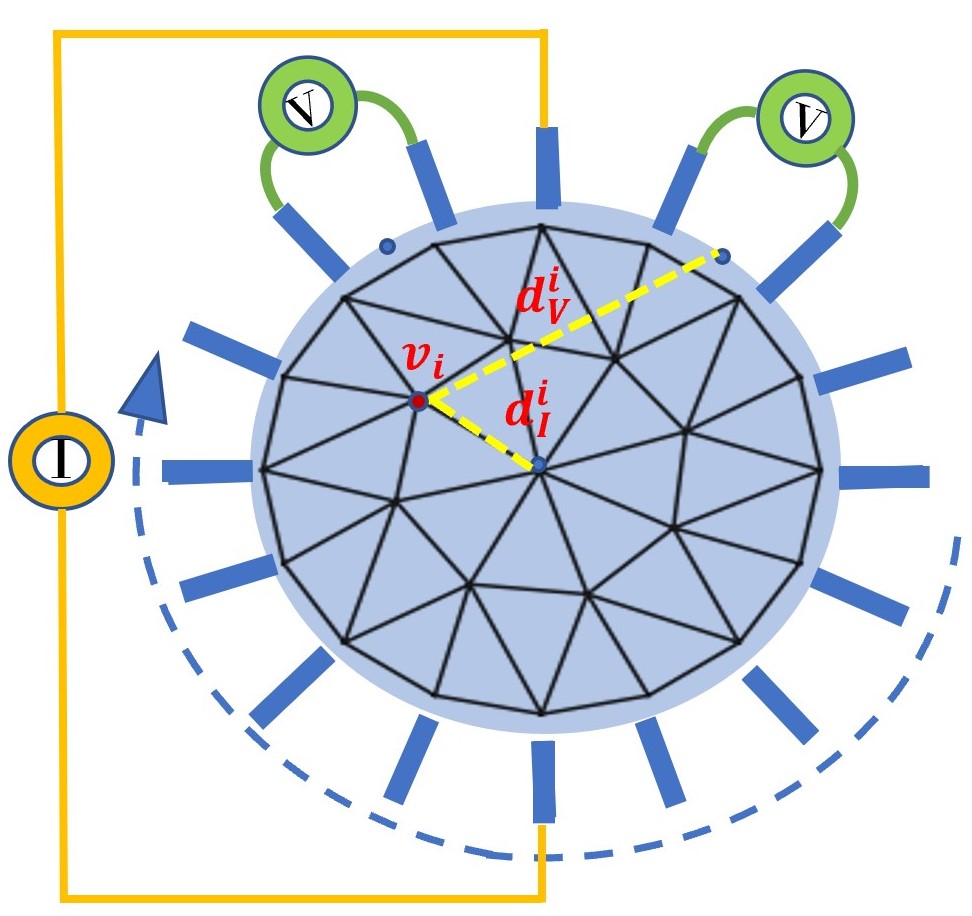} \\ 
        (a) & (b)\\
        \end{tabular}
    \caption{Oracle-Net, weight vector setup: (a)  adjacent injection-adjacent measurements; (b) opposite injection-adjacent measurements}
    \label{fig:ONETbis}
\end{figure}

Oracle-Net is a supervised GNN which in the training phase takes in input the exact support mask $M_{\mathcal{O}}$ for each training sample $j$ for $j=1,\dots, N$, 
and computes the optimal parameters $\Theta$ by minimizing the joint loss function: 
\begin{equation}
\mathcal{L}(\Theta):=\sum_{j=1}^{N} \big( \sum_i BCE( p_i(\Theta),[M_{\mathcal{O}}]_i^j) + \sum_i MSE(p_i(\Theta),[M_{\mathcal{O}}]_i^j) \big)
\end{equation}
where $BCE(q_1,q_2)$ is binary cross-entropy between the target GT ($q_2$) and the input probability ($q_1$) in one single class, and $MSE(q_1,q_2)$ applying an $L_2$ norm to minimize the difference between the ground truth $q_2$ and the probability $q_1$ obtained by the graph network at vertex $v_i$.

\section{Numerical Results}
\label{sec:sec7}

This section presents numerical results that illustrate
the performance of the proposed Oracle-Net-based Proximal Gradient Method for the nonlinear inverse EIT reconstruction problem.
In Section \ref{sec:sec71} we validate the Oracle-Net for the support estimation task; in Section \ref{sec:sec72} results on reconstruction performance on a 2D dataset are shown both in case of noise-free and noisy datasets; finally, in Section \ref{sec:sec73}
we address some experimental issues on the relation between the number of measurements and the sparsity in the Oracle mask, a well-assessed relation in linear compressed sensing that has not yet been established for the nonlinear case such as the EIT inverse problem.

The data used for the experimental session consists of a 2D synthetic EIT dataset generated on a mesh $\Tau_h$ composed of $n=1602$ vertices and $N_h=3073$ triangles. 
All examples simulate a circular tank slice of
unitary radius. In the circular boundary ring, $p=32$ equally spaced electrodes are located. 
The conductivity of the background liquid is set to be $\sigma_0 = 1.0$ $ \Omega m^{-1}.$
Each sample consists of a random number from 1 to 4 of inclusions inside the circular tank, localized randomly and characterized by random radius in the range $[0.15,0.25]$ and magnitude in the range $[0.2,2]$. 
Each inclusion consists of a homogeneous material with the same conductivity intensity. The region covered by the inclusions is significantly less than the total tank area, this leads to the sparsity in the solution vector $\sigma-\sigma_0$.

The acquisition of $m$ measurements is simulated through adjacent injection - adjacent measurement protocol for the results reported in Section \ref{sec:sec71}, and opposite injection - adjacent measurement protocol for the experiments illustrated in Sections \ref{sec:sec72} and \ref{sec:sec73}. In all the examples the setup
is considered blind, that is no a priori information about the sizes or locations of the inclusions is
considered.
In the forward calculations of $\Phi(\sigma)$ in \eqref{eq:regularization} we applied the KTCFwd forward solver, a two-dimensional version of the FEM described in \cite{Vetal1999}, kindly provided by the authors (website {\texttt https://github.com/CUQI-DTU/KTC2023-CUQI4}), which is based on a FEM implementation of the CEM model on triangle elements. The electric potential is discretized using second-order polynomial basis functions, while the conductivity is discretized on the nodes using linear basis functions on triangle elements.

\subsection{Oracle-Net validation}
\label{sec:sec71}

For the training of the Oracle-Net, we used an ad hoc designed dataset which consists of $5000$ instances each of dimension $992$. 
A portion of the dataset which is the $70\%$ of the total number is used for training,  the $15\%$ is used as a validation set and the remaining $15\%$ is employed as the test set for performance assessment.
Training of Oracle-Net has been performed with ADAM  optimizer, \cite{Adam}, using a learning rate equal to $2.5 \times 10^{-3}$ through $2034$ epochs with a mini-batch size of $10$ instances. 

As a figure of merit for assessing the performance of the Oracle-Net, we used the False Negative value $FN$, which represents the percentage on the test sample of misclassified vertices that corresponds to the vertices in the domain $\Tau_h$ belonging to inclusions which are misclassified as background.
For the success of the reconstruction algorithm, this represents the major drawback produced by the Oracle-Net since the misclassified vertices will not be carried out by the reconstruction algorithm.

In Fig.\ref{fig:OmaskIDEAL} we report the ideal Oracle mask (directly obtained by the ground truth conductivity distribution) superimposed to the mask $M_{\mathcal{O}}$ obtained by the Oracle-Net for six different samples.
For increasing threshold values $\sigma_{th}=\{0.4, 0.8,0.9 \}$ the support estimated by the Oracle-Net enlarges, while the $FN$ measure, reported in the bottom of each mask, decreases.
In Fig. \ref{fig:OmaskIDEAL}, in the first two columns the samples are obtained by an adjacent-adjacent protocol (see Fig.\ref{fig:ONETbis}(a)), while for the results in the other columns an opposite-adjacent protocol has been applied (see Fig.\ref{fig:ONETbis}(b)). The protocol indeed does not affect the quality of the obtained masks.

\begin{figure}[h]
    \centering
    \begin{tabular}{c|cccccc}
    {\begin{turn}{90}{$\sigma_{th} =0.4$}
    \end{turn} 
    } 
    &
\includegraphics[width=0.12\textwidth]{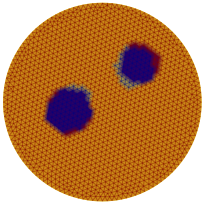} &
\includegraphics[width=0.12\textwidth]{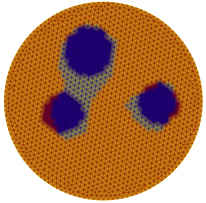} &
\includegraphics[width=0.12\textwidth]{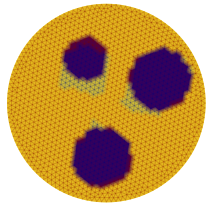}&
\includegraphics[width=0.12\textwidth]{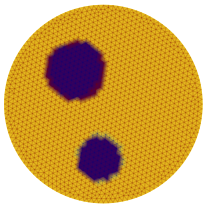}&
\includegraphics[width=0.12\textwidth]{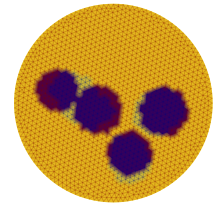} &
\includegraphics[width=0.12\textwidth]{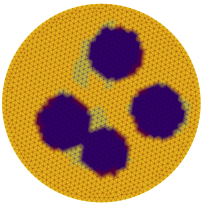}
\\
&{\tiny FN=4.67} &{\tiny FN=8.0} &{\tiny FN=2.37} & {\tiny FN=1.87}&{\tiny FN=5.37}  &{\tiny FN= 2.43}
\\
\hline
 { \begin{turn}{90}  $\sigma_{th} =0.8$
 \end{turn} } &
\includegraphics[width=0.12\textwidth]{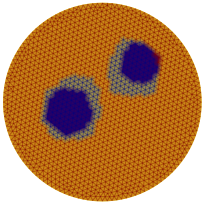} &
\includegraphics[width=0.12\textwidth]{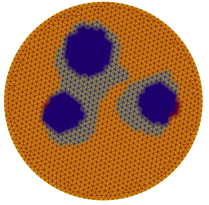} &
\includegraphics[width=0.12\textwidth]{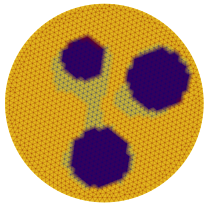} &
\includegraphics[width=0.12\textwidth]{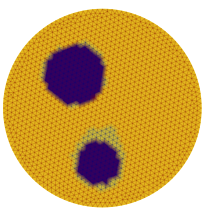}&
\includegraphics[width=0.12\textwidth]{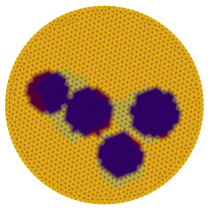} &
\includegraphics[width=0.12\textwidth]{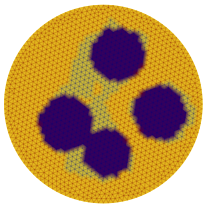}
\\
&{\tiny FN=1.43}& {\tiny FN=3.0} &{\tiny  FN=0.44} &{\tiny FN=0.06}&{\tiny FN=2.31}&{\tiny FN=0.37} \\
\hline
\\
 { {\begin{turn}{90}{$\sigma_{th} =0.9$}
    \end{turn} 
    } } &
\includegraphics[width=0.12\textwidth]{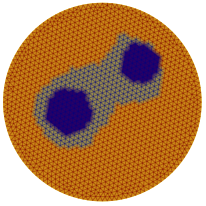} &
\includegraphics[width=0.12\textwidth]{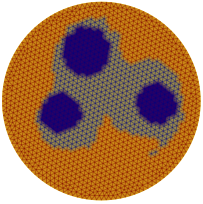} &
\includegraphics[width=0.12\textwidth]{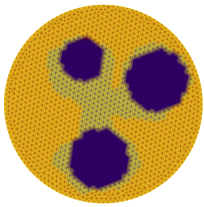}&
\includegraphics[width=0.12\textwidth]{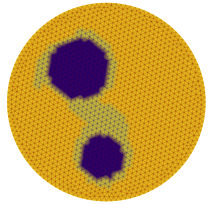}&
\includegraphics[width=0.12\textwidth]{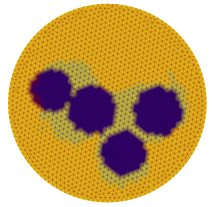} &
\includegraphics[width=0.12\textwidth]{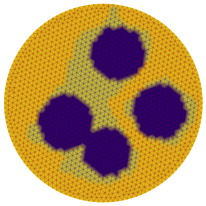}
\\
&{\tiny FN=0.81} &{\tiny FN=1.8} &{\tiny FN=0.06} &{\tiny FN=0.0} &{\tiny FN=0.62}  &{\tiny FN=0.0}\\
    \end{tabular}
    \caption{Ideal Oracle  (solid) overlapped to the learned Oracle mask $M_{\mathcal{O}}$ (in transparency). Top to bottom: three different thresholds $\sigma_{th}$ applied to six different samples.} 
    \label{fig:OmaskIDEAL}
\end{figure}

\begin{figure}[h]
    \centering
    \begin{tabular}{cc|cccc}
    \hline
    $M_{\mathcal{O}}$ &
    GT & PGM-TV-$M_{\mathcal{O}}$
    & PGM-TV &
    PGM-$\ell_1$-$M_{\mathcal{O}}$ &
    PGM-$\ell_1$ \\
    \hline 
\includegraphics[width=0.13\textwidth]{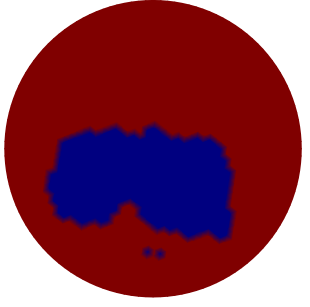} & \includegraphics[width=0.13\textwidth]{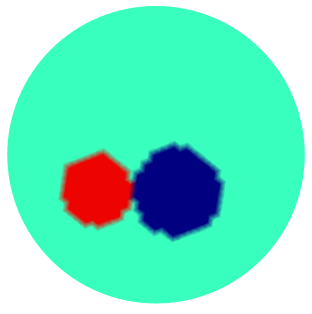} &
\includegraphics[width=0.13\textwidth]{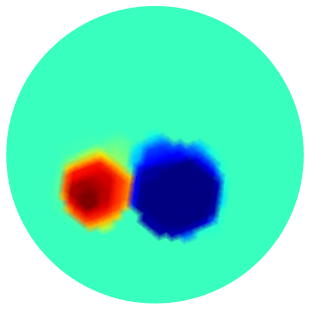} & \includegraphics[width=0.13\textwidth]{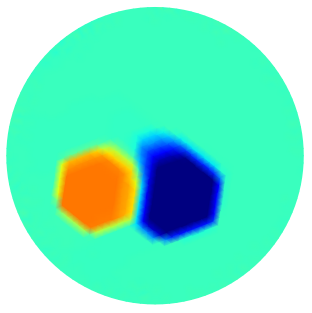} &    \includegraphics[width=0.13\textwidth]{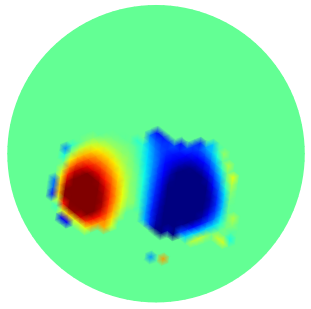} & \includegraphics[width=0.13\textwidth]{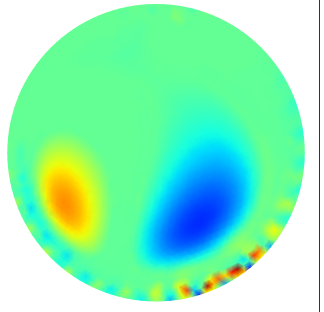} \\
\scriptsize{FN=0.0} \%& & \scriptsize{PSNR=32.84} & \scriptsize{PSNR=28.78} & \scriptsize{PSNR=27.19} &
\scriptsize{PSNR=22.04} \\
\hline
\includegraphics[width=0.13\textwidth]{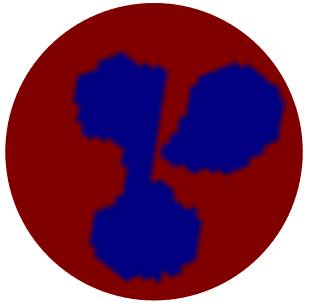} & \includegraphics[width=0.13\textwidth]{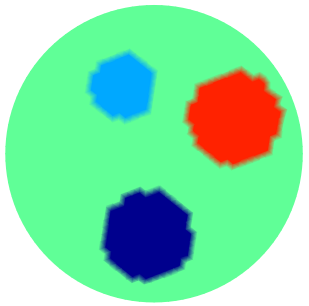} &
\includegraphics[width=0.13\textwidth]{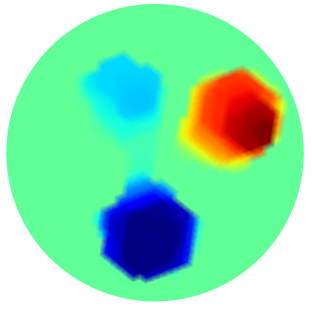} & \includegraphics[width=0.13\textwidth]{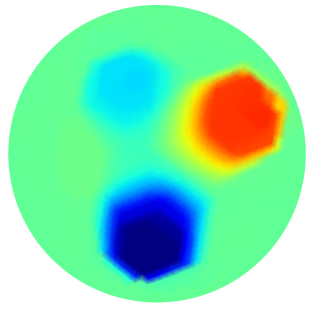} &    \includegraphics[width=0.13\textwidth]{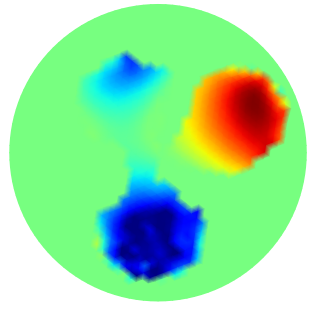} & \includegraphics[width=0.13\textwidth]{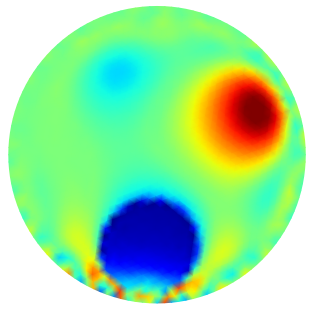}\\
\scriptsize{FN=0.24}\% & & \scriptsize{PSNR=24.555} & \scriptsize{PSNR=23.1313} & \scriptsize{PSNR=23.1387} &
\scriptsize{PSNR=20.051} \\
\hline
\includegraphics[width=0.13\textwidth]{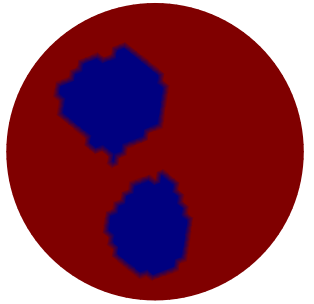}
& \includegraphics[width=0.13\textwidth]{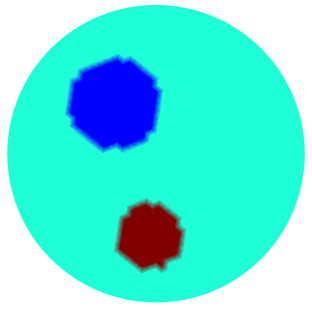} &
\includegraphics[width=0.13\textwidth]{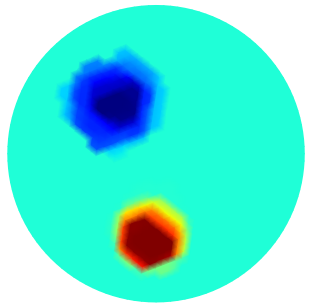} & \includegraphics[width=0.13\textwidth]{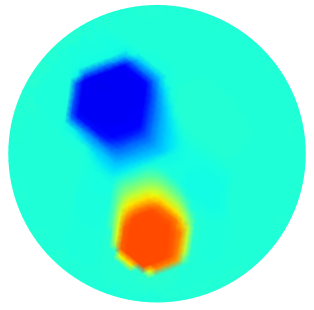} &    \includegraphics[width=0.13\textwidth]{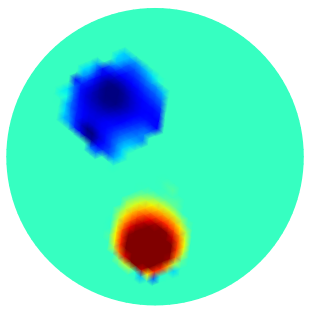} & \includegraphics[width=0.13\textwidth]{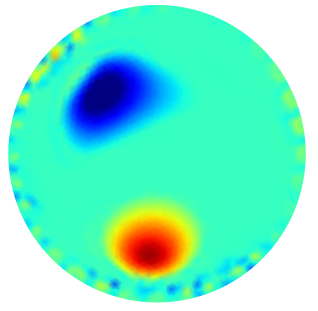}\\
\scriptsize{FN=0.0}\%& & \scriptsize{PSNR=27.53} & \scriptsize{PSNR=25.99} & 
\scriptsize{PSNR=26.93} &
\scriptsize{PSNR=22.74} \\
\end{tabular}
    \caption{Noise-free datasets, reconstructions with the different algorithms applying 
    opposite-adjacent protocol and $\sigma_{th}=0.8$ for the threshold of the Oracle-Net result.}
    \label{fig:Omask_NONOISE}
\end{figure}

\begin{figure}[h]
    \centering
    \begin{tabular}{cc|cc}
    \hline
    $M_{\mathcal{O}}$ & GT & PGM-TV-$M_{\mathcal{O}}$
    & PGM-TV \\
    \hline 
\includegraphics[width=0.13\textwidth]{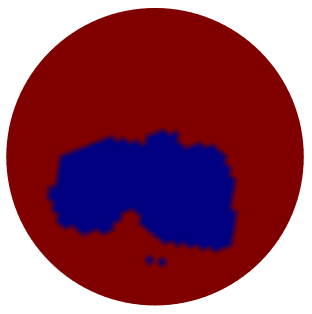} & \includegraphics[width=0.13\textwidth]{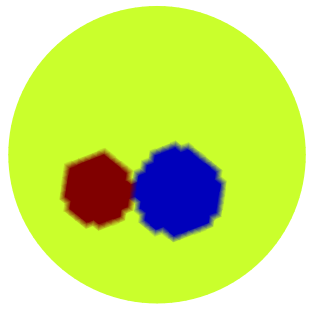} &
\includegraphics[width=0.13\textwidth]{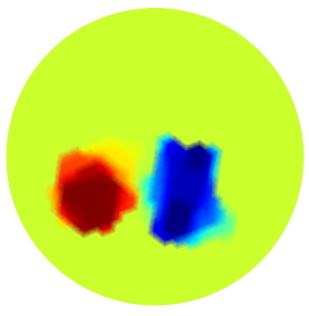} & \includegraphics[width=0.13\textwidth]{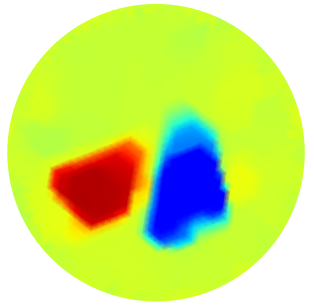} \\
\scriptsize{FN=0.0} \%& & \scriptsize{PSNR=26.10} & \scriptsize{PSNR=25.21} \\
\hline
\includegraphics[width=0.13\textwidth]{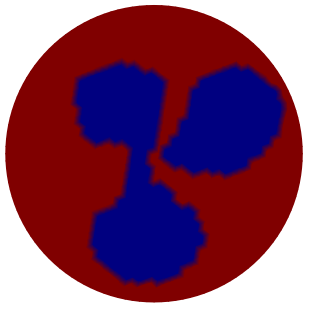} & \includegraphics[width=0.13\textwidth]{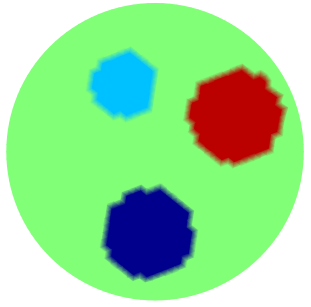} &
\includegraphics[width=0.13\textwidth]{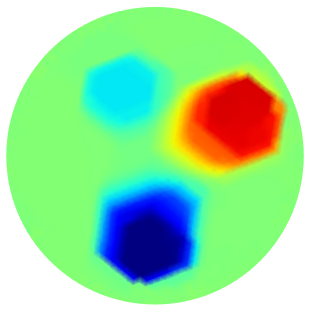} & \includegraphics[width=0.13\textwidth]{Noisy_Rec_No_Oracle_FB_TV_test_129.png} \\
\scriptsize{FN=0.62} \%& & \scriptsize{PSNR=23.1918} & \scriptsize{PSNR=23.044} \\
\hline
\includegraphics[width=0.13\textwidth]{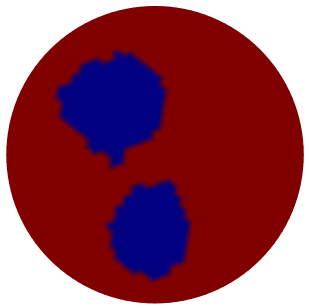} & \includegraphics[width=0.13\textwidth]{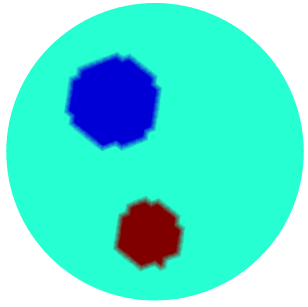} &
\includegraphics[width=0.13\textwidth]{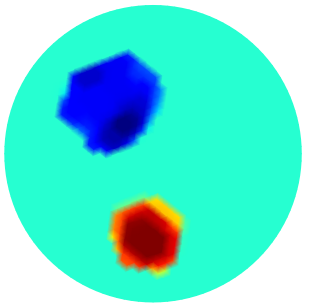} & \includegraphics[width=0.13\textwidth]{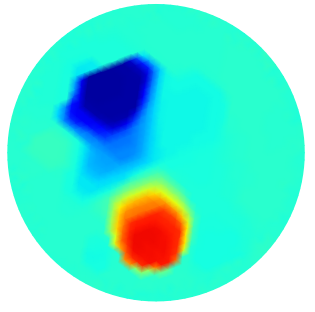} \\
\scriptsize{FN=0.0} \%& & \scriptsize{PSNR=26.999} & \scriptsize{PSNR=24.6806} \\
\hline
\end{tabular}
    \caption{Noisy dataset, reconstructions with the different algorithms applying 
    opposite-adjacent protocol  and $\sigma_{th}=0.8$ for the threshold of the Oracle-Net result.}
\label{fig:Omasknoise}
\end{figure}

\bigskip

\subsection{EIT reconstruction results}
\label{sec:sec72}

The inference phase of the Oracle-Net allows us to determine the estimated support mask $M_{\mathcal{O}}$ for a given set of new measurements (unseen from the Oracle). Then the conducibility $\sigma$ is computed by applying the PGM algorithm. 

In order to investigate the influence on the performance of the Oracle-Net in the EIT reconstruction problem \eqref{eq:regularization}, we compared the results of PGM-TV-$M_{\mathcal{O}}$
and PGM-$\ell_1$-$M_{\mathcal{O}}$, which exploit the Oracle-Net estimated support, with 
the standard PGM algorithms PGM-TV and  PGM-$\ell_1$ analogously regularized.
In all the reported  experiments, we selected the optimal $\lambda$ value through trials and errors, while the $\rho$ value  was set to $10^{-12}.$

First, we consider the case where the measurement error $\eta$ in \eqref{eq:cs} vanishes. 

In Fig.\ref{fig:Omask_NONOISE} we show some sample conductivity reconstructions and we report below the corresponding PSNR evaluation. For each sample, the first column represents the estimated support $M_{\mathcal{O}}$ pre-computed by the Oracle-Net together with the FN value associated, the second column illustrates the target ground truth (GT) reconstruction, and from the third to the sixth columns the computed reconstructions are shown for the different algorithms together with the PSNR values obtained.
From a visual inspection of the illustrated results and from the PSNR values reported, we can observe the benefit of incorporating the Oracle-Net estimated support into the PGM algorithmic framework. PGM-TV  and PGM-$\ell_1$ achieve lower performance than their Oracle-Net-based counterparts. Moreover, the use of the TV regularizer seems to be beneficial with or without Oracle exploitation.

\bigskip

We finally evaluate the robustness against measurements corrupted by additive noise. The measured voltage is computed by the forward model KTCFwd and recorded as a vector $V_m=\Phi(\sigma) \in \R^m$. According to the degradation model \eqref{eq:cs}, the noiseless measures $V_m$ are corrupted by additive white Gaussian noise to simulate experimentally measured voltages, with noise level $\bar{\delta}$:
\begin{equation}
\eta =  \bar{\delta} \; \|{V_m}\|_2  \; \bar{n}, \quad \bar{n} \in \mathcal{N}(0,1).
\end{equation}
The Signal to Noise Ratio (SNR)
in dB is calculated as 
$SNR = 10 \log_{10} (\frac{\| V_m \|_2^2}{\|\eta\|_2^2})$, while Peak-SNR (PSNR) will be used to measure the quality of an image after the reconstruction.

In Fig.\ref{fig:Omasknoise} some conductivity reconstructions are illustrated. The  collected voltage has been corrupted by adding a realization of random noise with Gaussian distribution and noise level 
 $\bar{\delta}=2.5 \times 10^{-3}$ (added noise is such that the intrinsic SNR is 40dB).
We observe that, even though the Oracle-Net was trained on noiseless data, the mask obtained in inference $M_{\mathcal{O}}$ with noise measurements -  Fig.\ref{fig:Omasknoise} first column - is sufficiently accurate, and PGM algorithm performs very well on the noise measurements.
A natural consequence of having noise-corrupted measurements is the degradation in the conductivity reconstruction. Comparing the results in the third columns of Fig.\ref{fig:Omask_NONOISE}  and Fig.\ref{fig:Omasknoise} we observe quantitatively a degradation in the performance of reconstructions in terms of PSNR values. Analogous behavior can be observed comparing the fourth columns of  Fig. \ref{fig:Omask_NONOISE}  and Fig. \ref{fig:Omasknoise}, which have been obtained without taking advantage of the contribution of the Oracle Mask. 
\begin{figure}[ht]
\centering
\includegraphics[width=0.4\textwidth]{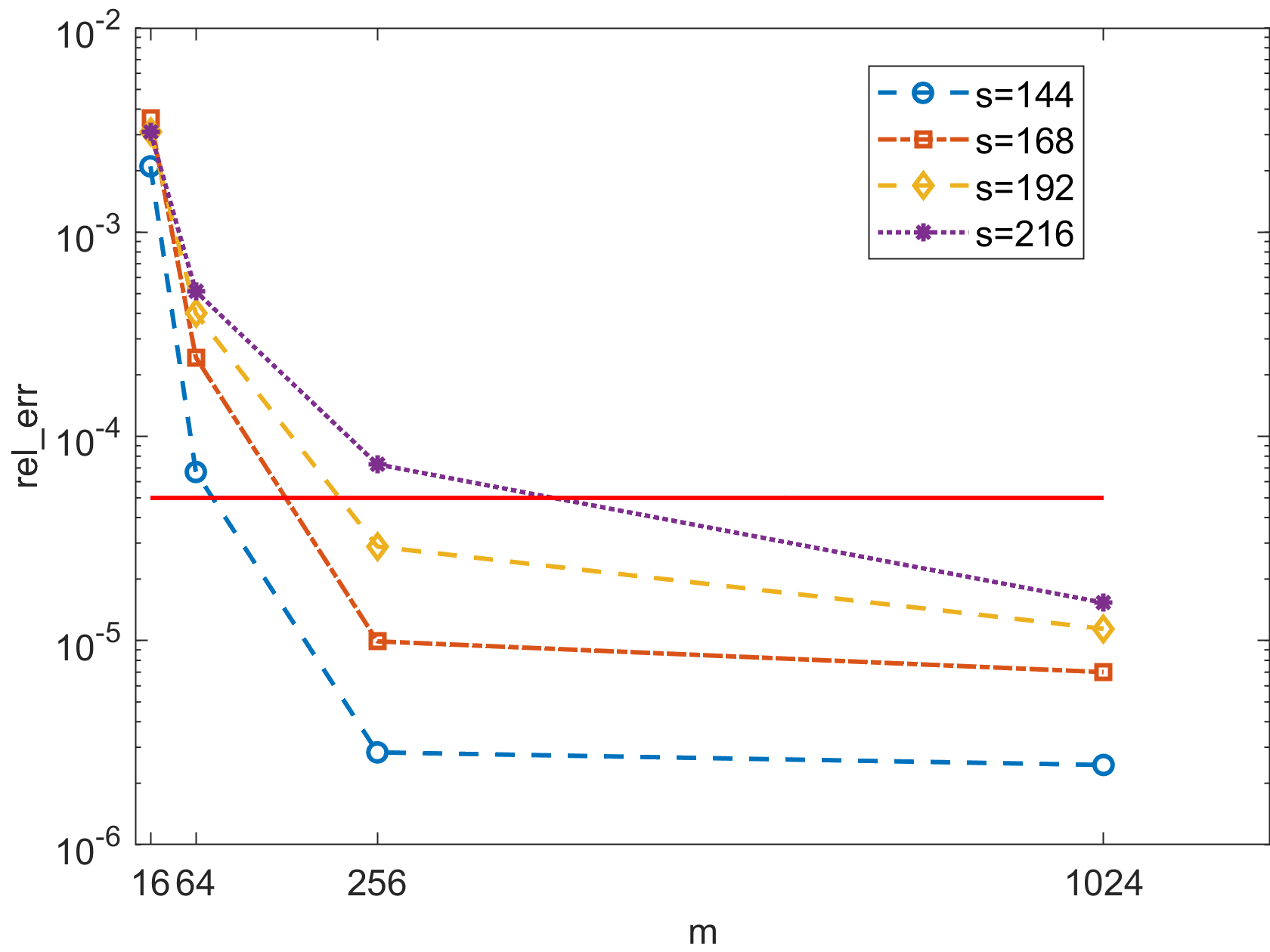}
\includegraphics[width=0.4\textwidth]{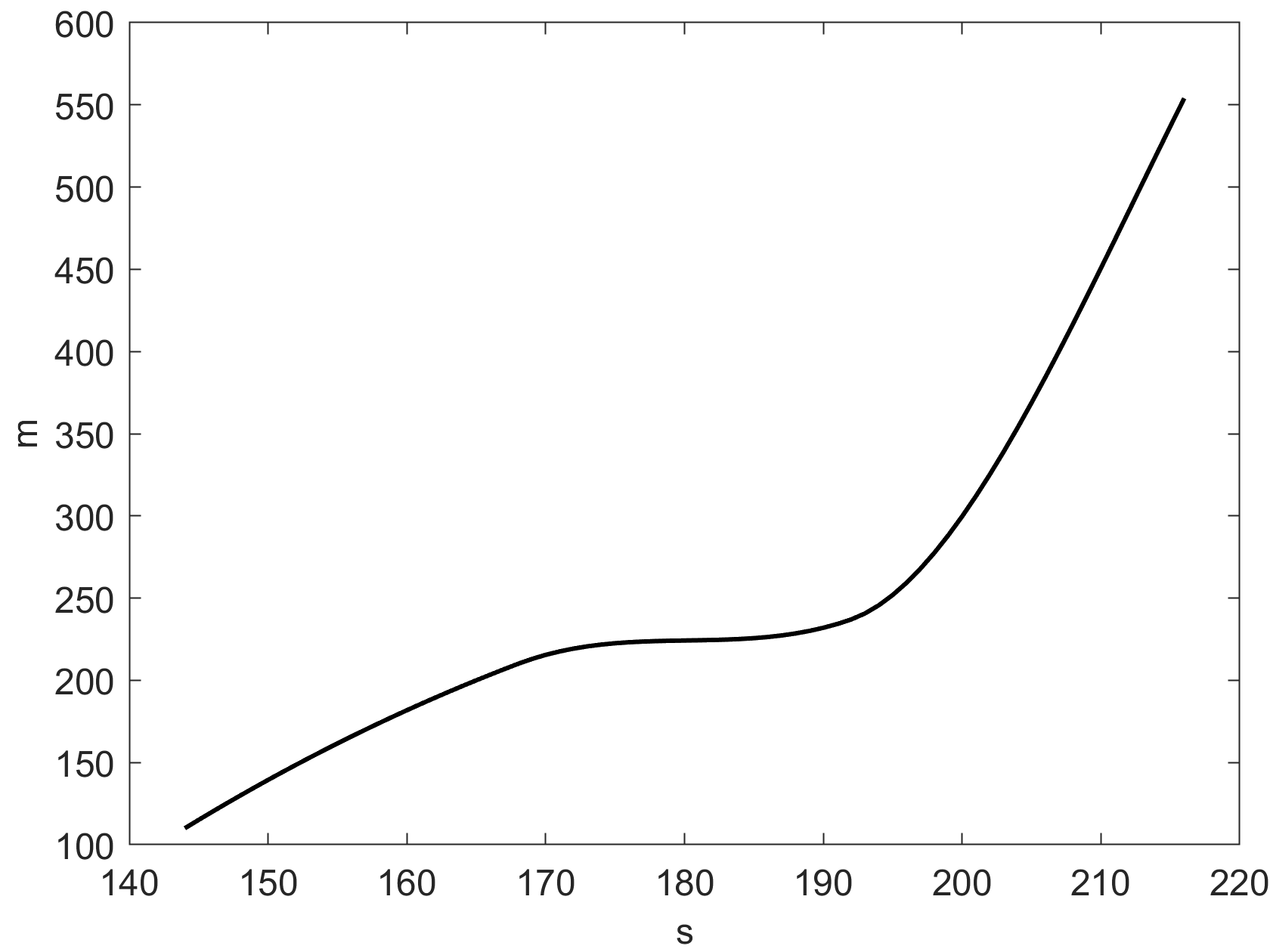}
\caption{Left - Plot of the relative errors in the recovery of four samples $\sigma$ with increasing gradient $s$-sparsity obtained by PGM-TV-$M_{\mathcal{O}}$ with an increasing number of measurements $m$. Right - Plot of the empirical relation between the number of measurements $m$ and the sparsity factor $s$. 
} 
\label{Fig_CS}
\end{figure}

\subsection{Experimental compressed sensing issues}
\label{sec:sec73}

A well-known result in well-posed linear CS asserts that when $\sigma \in \R^n $ is $s$-sparse, the
recovery via $\ell^1$-minimization is provably exact, using at least a number of measurements $m$ is roughly $O(s \log n)$. To the best of our knowledge, a similar result has not been achieved yet in the case of nonlinear measurements.
To conclude this numerical session we would like to investigate, at least from an experimental point of view, the relation between the $s$-sparsity factor, which characterizes the conductivity $\sigma$ to be reconstructed, and the number of nonlinear measurements $m$, needed to obtain a suboptimal recovery. 

The quality of the reconstructed conductivity $\sigma$ is strictly related to the quantity $m$ and the quality of the acquired measurements $V_M$.  In this simplified context we are aware of neglecting many other factors that could affect the reconstruction accuracy, i.e., the mismodeling of the domain, or the misplacement of electrodes. 
However, it is clear that obtaining high-quality EIT reconstructions with a reduced number of electrodes (and thus of measurements) would be of great help to reduce the costs and increase the reliability of EIT systems in practical applications. 
It would be even more useful to know the sufficient number of measurements needed to recover an optimal $\sigma$, under sparsity conditions on the unknown conductivity. 
However, unlike the well-posed CS linear context, we address the compressed sensing recovery problem in an ill-posed setting where the observations are nonlinear. 
 We aim at least to show and analyze the experimental relationship between the number of measurements $m$ needed for a suboptimal recovery and the $s$-sparsity of the unknown conductivity distribution.
At this aim we applied PGM-TV-$M_{\mathcal{O}}$ algorithm to the reconstruction of four conductivity samples with increasing gradient sparsity, using an increased number of measurements.
In the idealized set-up, each sample is noise-free and contains a single circular inclusion, with a different ray; for increasing ray values the sample represents a decreasing $s$-sparsity value, characterized by $s:=||\nabla \sigma||_0  = \{144,168,192,216\}$.

In Fig. \ref{Fig_CS}, left panel, the relative $\ell^2$ reconstruction errors $ rel\_err := \| \sigma-\sigma^* \|_2/ \| \sigma\|_2$ are shown 
for each sample in terms of  the number of measurements $m$ used for solving the inverse EIT reconstruction problem, with $m$ that varies in the range  $m=\{ 16, 64,256,1024\}$.
 
The horizontal solid red line indicates the error threshold below which the reconstruction can be considered suboptimal (with relative error $5\times 10^{-5}$ which corresponds to PSNR=90). As expected, when the gradient-sparsity decreases, which corresponds to increasing values of $s$, the number of measurements necessary to obtain a suboptimal reconstruction increases. 
Finally, it should be emphasized that, for samples with a more severe $s$-sparsity, beyond a certain threshold, the improvement in reconstruction, in the face of an increase in the number of measurements (from $m=256$ to $m=1024$), is no longer significant. 
The plot in Fig.\ref{Fig_CS}, right panel, represents the empirical relation between the number of measurements $m$ and the sparsity factor $s$, obtained by interpolating the points in the plot - Fig. \ref{Fig_CS}(left panel) - which intersect the optimal recovery line (in solid red). The adopted CS strategy allows us to reconstruct the conducibility distributions efficiently, even when the number of measurements is much smaller than the data’s dimension. Thus the plot would help to determine the sufficient number of measurements needed to recover an optimal $\sigma$, under sparsity conditions on the unknown conductivity.

\section{Conclusions}
\label{sec:sec8}
This paper demonstrated a proof of concept study in using CS techniques in the numerical solution of nonlinear ill-posed inverse problems.  We proposed a sparsity-aware PGM for the solution of a variational formulation of the EIT inverse problem.   The sparsity inducing role is taken by a new concept of “Oracle” which infers the optimal support for a given set of nonlinear measurements. By exploiting the sparsity or compressibility of the signal distribution, CS reduces the amount of data needed for accurate reconstruction. The Oracle is designed by an autoencoder GNN that automatically predicts a binary mask which localizes the inclusions thus reducing memory requirements and processing time while maintaining recovery accuracy. The accurate recovery is demonstrated, using the proposed  sparsity-aware PGM algorithm, under the requirements that the Jacobian of the measurement system $\Phi$  satisfies a RIP-like condition and that $\Phi$ is mildly nonlinear.
Moreover, we shed light on the problem of determining how few measurements suffice for an accurate EIT sparsity-regularized reconstruction, a well-known result in well-posed linear CS. An interesting future direction will address this issue. Finally, we will consider other interesting nonlinear CS contexts where the proposed Oracle-based strategy can be successfully applied.

\section*{Data Availability} Data will be made available on request.

\section*{Declarations}
\noindent \textbf{Conflict of interest} The authors declare no conflict of interest.

\section*{Acknowledgments}

This work was supported in part by the National Group for Scientific Computation (GNCS-INDAM),
Research Projects 2024, and in part by MIUR RFO projects.
The research of LR has been funded by PNRR - M4C2 - Investimento 1.3. Partenariato Esteso PE00000013 - ``FAIR - Future Artificial Intelligence Research'' - Spoke 8 ``Pervasive AI'', which is funded by the European Commission under the NextGeneration EU programme. LR also acknowledges the support of ``Gruppo Nazionale per l’Analisi Matematica, la Probabilità e le loro Applicazioni'' of the ``Istituto Nazionale di Alta Matematica'' through project GNAMPA-INdAM 2023, code CUP\_E53C22001930001. The work of SM and DL was supported by PRIN2022\_MORIGI, titled "Inverse Problems in the Imaging Sciences (IPIS)" 2022 ANC8HL - CUP J53D23003670006, and PRIN2022\_PNRR\_CRESCENTINI CUP J53D23014080001. 
\bibliographystyle{abbrv}
\bibliography{refs}

\newpage

\appendix
\section{Appendix}

\noindent{\bf Proof of Theorem \ref{thm:blumensath}} 
We follow the main ideas behind the proof of \cite[Theorem 2]{blumensath2013compressed}, with some modifications due to use of the scheme \eqref{eq:PG-Phi} instead of the Iterative Hard Thresholding, studied therein, and a different treatment of the bounds in \eqref{eq:RIP}, leading to different conditions that must be satisfied by the parameters.
We start by proving the following claim:
\begin{equation}
\| \sigmanp - \sigma \|^2 \leq (1-\mu \lambda \rho - \mu \alpha + \mu \gamma) \| \sigman - \sigma \|^2 + 4 \mu \mathcal{J}_\lambda^\delta(\sigma),
    \label{eq:claim}
\end{equation}
which holds for any $\sigma \in \operatorname{dom}(g)$. 
Indeed, let us first consider the definition of $\sigmanp$, which is given by \eqref{eq:PG-Phi} and can be rewritten as follows (to simplify the expressions, we denote by $J$ the matrix $J_\Phi(\sigman)$)
\[
\sigmanp = \argmin_{\sigma \in \R^\nt} \left\{ 
\frac{1}{2} \| \sigma - \sigman + \mu \lambda \rho \sigman - \mu J^T(\Ld- \Phi(\sigman))\|^2 + \mu \lambda g(\sigma)
\right\};
\]
therefore, due to the convexity of $g$, the minimizer $\sigmanp$ satisfies the following optimality conditions:
\[
- (\sigmanp - \sigman + \mu \lambda \rho \sigman - \mu J^T(\Ld-\Phi(\sigman))) \in \mu \lambda \partial g(\sigmanp),
\]
which translates into 
\[
\mu \lambda  (g(\sigma) - g(\sigmanp)) \geq - \langle \sigmanp - \sigman + \mu \lambda \rho \sigman - \mu J^T(\Ld-\Phi(\sigman)), \sigma - \sigmanp)  \rangle \qquad \forall \sigma \in \operatorname{dom}(g).
\]
Using simple algebraic manipulations, we get
\[
\begin{aligned}
    &\frac{1}{2}\| \sigmanp - \sigma \|^2 + \frac{1}{2} \| \sigmanp - \sigman + \mu \lambda \rho \sigman - \mu J^T(\Ld-\Phi(\sigman)) \| \\
    & \qquad - \frac{1}{2} \| \sigma - \sigman + \mu \lambda \rho \sigman - \mu J^T(\Ld-\Phi(\sigman))\|^2 \leq \mu \lambda g(\sigma) - \mu \lambda g(\sigmanp).
\end{aligned}
\]
Neglecting positive and negative terms on the left and right-hand side, respectively (using the non-negativity of $g$), we obtain
\begin{equation}
 \| \sigmanp - \sigma \|^2 \leq \| \sigma - \sigman + \mu \lambda \rho \sigman - \mu J^T(\Ld-\Phi(\sigman))\|^2 + 2 \mu \lambda g(\sigma)
    \label{eq:aux2}
\end{equation}
Let us now focus on the first term on the right-hand side of \eqref{eq:aux2}:
\[
\begin{aligned}
    \| \sigma - \sigman + \mu \lambda \rho \sigman - \mu J^T(\Ld-\Phi(\sigman))\|^2 = & \| \sigma - \sigman \|^2 + \mu^2 \| \lambda \rho \sigman - J^T(\Ld - \Phi(\sigman)) \|^2 \\ & + 2\mu\lambda\rho \langle \sigman, \sigma - \sigman \rangle - 2\mu \langle J(\sigma - \sigman), \Ld -\Phi(\sigman) \rangle \\
    =& \| \sigma - \sigman \|^2 + \mu^2\ \circled{1}
    + \mu\lambda\rho \ \circled{2}
    + \mu\ \circled{3}
\end{aligned}
\]
We furthermore observe that
\[
\begin{aligned}
\circled{1} &=  \| \lambda \rho \sigman - J^T(\Ld - \Phi(\sigman)) \|^2 \leq  2 \lambda^2 \rho^2 \| \sigman \|^2 + \| J^T(\Ld - \Phi(\sigman)) \|^2 \\
&\leq  2 \lambda^2 \rho^2 \| \sigman \|^2 + 2 \beta \| \Ld - \Phi(\sigman) \|^2,
\end{aligned}
\]
where used the upper bound in \eqref{eq:RIP} since $\sigman \in \operatorname{dom}(g) \subset K_{0,1}$; instead,
\[
    \circled{2} = 
    2 \langle \sigman, \sigma - \sigman \rangle = \| \sigma\|^2 - \| \sigman\|^2 - \| \sigma - \sigman \|^2
\]
and, analogously,
\[
\begin{aligned}
\circled{3} & = - 2 \langle J(\sigma - \sigman), \Ld -\Phi(\sigman) \rangle \\ &= \| \Ld - \Phi(\sigman) - J(\sigma - \sigman)\|^2 - \| J(\sigma - \sigman)\|^2 - \| \Ld - \Phi(\sigman) \|^2 \\
    & \leq \| \Ld - \Phi(\sigman) - J(\sigma - \sigman)\|^2 - \alpha \|\sigma - \sigman\|^2 - \| \Ld - \Phi(\sigman) \|^2,
\end{aligned}
\]
where we also used the lower bound in \eqref{eq:RIP}. 
The term $\Ld - \Phi(\sigman) - J(\sigma - \sigman)$ can be bounded as follows:
\[
\begin{aligned}
\| \Ld - \Phi(\sigman) - J(\sigma - \sigman) \|^2 \leq & 2 \| \Ld -\Phi(\sigma)\|^2 + 2 \| \Phi(\sigma) - \Phi(\sigman) - J(\sigma -\sigman)\|^2 \\
\leq & 2 \| \Ld -\Phi(\sigma)\|^2 + 2\gamma \| \sigma - \sigman \|^2,
\end{aligned}
\]
where we have used \eqref{eq:nonlin_ass} with $\sigma_1 = \sigma$ and $\sigma_2 = \sigman$. 
Collecting all the results in \eqref{eq:aux2}, we get
\[
\begin{aligned}
    \| \sigmanp - \sigma \|^2 \leq & \| \sigma - \sigman \|^2 + 2 \mu^2 \big( \lambda^2 \rho^2 \| \sigman\|^2 + \beta  \| \Ld - \Phi(\sigman) \|^2 \big) \\
    & + \mu\lambda\rho \big( \| \sigma\|^2 - \| \sigman\|^2 - \| \sigma - \sigman \|^2 \big) \\
    & + \mu \big( 2\| \Ld -\Phi(\sigma)\|^2 + 2\gamma \| \sigma - \sigman \|^2 - \alpha \|\sigma - \sigman\|^2 - \| \Ld - \Phi(\sigman) \|^2 \big) \\
    & + 2\mu \lambda g(\sigma) \\
    = & (1-\mu \lambda \rho -\mu \alpha + 2\mu \gamma) \| \sigman  - \sigma \|^2 + (2 \mu^2 \lambda^2 \rho^2 - \mu \lambda \rho) \|\sigman\|^2 \\
    & + (2\mu^2\beta - \mu) \| \Ld - \Phi(\sigman) \|^2 + \mu \lambda \rho  \| \sigma\|^2 + 2 \mu \| \Phi(\sigma) - \Ld\|^2 + 2\mu \lambda g(\sigma).
\end{aligned}
\]
Imposing that $2 \mu^2 \lambda^2 \rho^2 - \mu \lambda \rho \leq 0$ (i.e., $\mu \leq \frac{1}{2\lambda \rho}$) and $2\mu^2\beta -\mu \leq 0$ (i.e., $\mu \leq \frac{1}{2\beta}$), we finally retrieve \eqref{eq:claim}, where we used again the non-negativity of $g$. Let us now define $q = 1-\mu \lambda \rho -\mu \alpha + 2 \mu \gamma$: then, applying \eqref{eq:claim} recursively, we get
\[
\| \sigmanp - \sigma \|^2 \leq q^n \| \sigma^{(0)} - \sigma \|^2 + 4\mu \mathcal{J}_\lambda^\delta(\sigma) \sum_{i=0}^n q^i.
\]
This shows that, if we impose that $q < 1$ (which motivates the last bound in \eqref{eq:parameters}), the sequence $\{\sigman\}$ is bounded and convergent to a cluster point $\overline{\sigma}$ such that
\[
\| \overline{\sigma} - \sigma \|^2 \leq \frac{4\mu}{1-q} \mathcal{J}_\lambda^\delta(\sigma) ,
\]
and substituting the expression of $q$ we recover \eqref{eq:convergence}.

\end{document}